\documentclass[journal]{IEEEtran}

\usepackage[utf8]{inputenc}
\usepackage[english]{babel}
\usepackage{hyperref}
\usepackage{mathtools}

\usepackage{amsmath,dsfont,amsthm,cancel,
amsfonts,amssymb,mathrsfs,amsthm,cmll,nicefrac}
\usepackage{textcomp}
\usepackage{graphicx}
\usepackage{algorithm}
\usepackage{algorithmic} 
\usepackage{color}

\usepackage{enumitem}
\usepackage{graphics}
\usepackage{lineno}
\usepackage{hyperref}
\usepackage{MnSymbol,stmaryrd} 
\usepackage{manfnt}

\usepackage{tikz}
\usetikzlibrary{automata} 
\usepackage[caption=false,font=normalsize,labelfon
t=sf,textfont=sf]{subfig}
\usepackage{tikz}
\usetikzlibrary{shapes,shadows,backgrounds}

\usepackage{enumitem}

\ifCLASSINFOpdf
\else
\fi

\DeclareMathOperator{\Me}{Me}
\newcommand{\ITEM}[1]{\noindent{#1}}
\providecommand{\tamano}[2]{\fontsize{#1}{28}\selectfont{#2}}
\providecommand{\lcorte}[2]{l^{\ast}_{#1}(\alpha_{#2})}
\providecommand{\rcorte}[2]{r^{\ast}_{#1}(\alpha_{#2})}
\providecommand{\crispy}[1]{\widetilde{#1}}

\providecommand{\corteai}[2]{\nicefrac{\large{#1}}{\alpha_{#2}}}
\providecommand{\cortea}[1]{\nicefrac{\Large{#1}}{\alpha}}

\providecommand{\operador}[3]{#1:#2\longrightarrow #3}
\providecommand{\con}[1]{\mathds{#1}}
\providecommand{\talque}[0]{\; : \;}
\providecommand{\Sup}[1]{\underset{#1}{\sup~}}
\providecommand{\NFR}[0]{\mathcal{F}(\mathds{R})}

\providecommand{\NFRN}[1]{\mathcal{F}_{-}(\mathds{R})}
\providecommand{\IR}[0]{\con{I}\con{R}}

\newcommand{\prt}[1]{\langle #1\rangle}

\newtheorem{teorema}{Theorem}[section]
\newtheorem{proposicion}{Proposition}[section]
\newtheorem{corolario}{Corollary}[section]

\newtheorem{definicion}{Definition}[section]
\newtheorem{observacion}{Remark}[section]

\newtheorem{ejemplo}{Example}[section]

\begin{document}

\title{Admissible orders on fuzzy numbers} 

\author{\IEEEauthorblockN{Nicol\'as Zumelzu
\IEEEauthorrefmark{2}, Benjam\'in Bedregal\IEEEauthorrefmark{3}, \IEEEmembership{Member,~IEEE,}
		Edmundo Mansilla		
		\IEEEauthorrefmark{2}, Humberto Bustince\IEEEauthorrefmark{4}, \IEEEmembership{Senior Member,~IEEE,}} and Roberto D\'iaz\IEEEauthorrefmark{5}
\thanks{
	\IEEEauthorrefmark{2}Departamento de Matem\'atica y F\'isica, Facultad de Ciencias, Universidad de Magallanes, Punta Arenas, Chile.
	\IEEEauthorrefmark{3}Departamento de Inform\'atica e Matem\'atica Aplicada - DIMAp, Universidade Federal do Rio Grande do Norte - UFRN, Natal - RN, 59.072-970, Brazil,
	\IEEEauthorrefmark{4}Institute of Smart Cities, Departamento  de Matematica, Estadistica e Informatica, Universidad Publica de Navarra - UPNA, Campus Arrosadia sn, Pamplona, Navarra, Pamplona, Navarra, 31006, Spain,
	\IEEEauthorrefmark{5}Departamento de Ciencias Exactas, Universidad de los Lagos, Osorno, Chile,
	 E-mail addresses: nicolas.zumelzu@umag.cl, bedregal@dimap.ufrn.br, edmundo.mansilla@umag.cl, bustince@unavarra.es, roberto.diaz@ulagos.cl}}

\markboth{Journal of \LaTeX\ Class Files,~Vol.~14, No.~8, August~2015}%
{Shell \MakeLowercase{\textit{et al.}}: Bare Demo of IEEEtran.cls for IEEE Journals}

\maketitle

\begin{abstract}
From the more than two hundred partial orders for fuzzy numbers proposed in the literature, only a few are total. In this paper, we introduce the notion of admissible order for fuzzy numbers equipped with a partial order, i.e. a total order which refines the partial order. In particular, it is given special attention to the partial order proposed by Klir and Yuan in 1995. Moreover, we propose a method to construct admissible orders on fuzzy numbers in terms of linear orders defined for intervals considering a strictly increasing upper dense sequence, proving that this order is admissible for a given partial order. Finally, we use admissible orders to ranking the path costs in fuzzy weighted graphs.
\end{abstract}

\begin{IEEEkeywords}
Fuzzy numbers, orders on fuzzy numbers, admissible orders, fuzzy weighted graphs.
\end{IEEEkeywords}

\IEEEpeerreviewmaketitle

\section{Introduction}

\IEEEPARstart{F}{uzzy} numbers have been introduced by Zadeh \cite{zadeh1965fuzzy} to deal with imprecise numerical quantities in a practical way. 
The concept of a fuzzy number plays a fundamental role in formulating quantitative fuzzy variables, i.e. variables whose states are fuzzy numbers. 

 The study of admissible orders   over the set of closed subintervals of $[0,1]$, i.e. orders which refine the natural order for intervals, starts with the work of Bustince \textit{et al.} \cite{bustince2013generation} and from then several {pieces of research} on this topic have been made, for example in \cite{Santana2020,Zapata17}. {Lately}, this notion was adapted for other domains in \cite{Ivanosca2016,Annax20,Laura16a,Laura17,matzenauer2021strategies}. 
 
From the more than two hundred partial orders for fuzzy numbers proposed in the literature, only a few are total, for example \cite{TAsmus2017,valvis2009new,wang2014total}. Moreover, no study on admissible orders for fuzzy numbers or a subclass of them has been made so far. In order to overcome this lack and motivated mainly by the application potential of this subject, in this work we introduce and analyze the notion of admissible orders for fuzzy numbers with respect to a partial order and in particular, we explore the case where this partial order is the given in \cite{Klir1995}.

On the other hand, fuzzy weighted graphs are a generalization of the weighted graphs where fuzzy numbers are used to model the uncertainty in the weigths of the edges (c.f. \cite{Cornelis2004}). The fuzzy shortest path problem was first enuntiated in \cite{DuboisPrade1980} and since then several algorithms have been proposed to determine the fuzzy shortest path length in fuzzy weighted graphs (see for example \cite{Cornelis2004,Klein1991,LinLWu2021}). It is worth to note that the order considered on fuzzy numbers is fundamental to such algorithms. In \cite{Cornelis2004} it is formalized and proposed an algorithm to  determine  the fuzzy shortest path (routes) length on fuzzy weighted graphs. It pays special attention to the ranking methods of the routes, based in a defuzzification method. Nevetheless, the approach presents a problem with the center of gravity defuzzication method. In this paper, we present a solution to what was raised in \cite{Cornelis2004}. It does not consider defuzzification thanks to the given definition of admissible order. 

This paper is organized as follows: In Section 2, in addition to establishing the notation used, we recall some essential notions for the remaining sections. In Section 3 we see the most basic partial order on fuzzy numbers, and a total order proposal in \cite{wang2014total}. The notion of admissible order for fuzzy numbers is studied in Section 4. Section 5, is presents an application of admissible orders in the Shortest Path problem. Section 6, we present another application in graphs, this time for the travelling salesman problem considering the capitals of the Brazilian Northeast. Finally, Section 7 provides some final remarks.

\section{Preliminary Concepts}

In this section, we introduce notations, definitions and preliminary facts which are used throughout this work. 

Given a poset $\prt{P,\leq}$ and $a,b\in P$, we denote by $a\parallel b$ when $a$ and $b$ are incomparable, i.e. when neither $a\leq b$ nor $b\leq a$. We will denote the set of real numbers by $\con{R}$.

\begin{definicion}\label{uppercontinuidad}  Let $X\subseteq  \con{R}$. The function $\operador{f}{X}{\con{R}}$ is upper semi-continuous, if for every $a\in X$ and $\varepsilon>0$, there exists $\delta>0$ such that when, $|x-a|<\delta$ for a $x\in X$ then $f(x)<f(a)+\varepsilon$.
\end{definicion}

This definition of \textit{upper semi-continuous} is not the same, but  is equivalent to the definition  given in \cite{yeh2006real}, considering the usual topology for real numbers.

Based on \cite{Klir1995,Fodor2004} we consider the following definitions of left and right continuity of unary increasing/decreasing real functions.

\begin{definicion}
Let $a\in \con{R}$, $f:{]-\infty,a[}\rightarrow \con{R}$ be  a decreasing  function, $g:{]a,+\infty[}\rightarrow \con{R}$ be an increasing function { and $\con{Z}^+$ be the set of positive integers}. Then $f$ is left-continuous if for each increasing sequence $(x_i)_{i\in \con{Z}^+}$ of real numbers {$x_i \to a$, as $n\to \infty$} we have that $\lim\limits_{i\rightarrow \infty} f(x_i) =f(\lim\limits_{i\rightarrow \infty} x_i)$.  Dually, $g$ is right-continuous if for each decreasing sequence $(x_i)_{i\in \con{Z}^+}$ of real numbers {$x_i \to a$, as $i\to \infty$}  we have that $\lim\limits_{i\rightarrow \infty} g(x_i) =g(\lim\limits_{i\rightarrow \infty} x_i)$. 
\end{definicion}

\subsection{Admissible Orders on the Real Closed Interval Set }

Let  $\con{I}\con{R}$ be the set of all the closed intervals of real numbers, i.e.
$$\con{IR}=\{[a,b] \talque a,b\in\con{R},~a\leq b\}.$$ 
Closed intervals of real numbers will be called just of intervals.
Degenerate intervals, that is, intervals $[a,a]$ will be written {in the} simplified form $[a]$.
Given an interval $A$, its lower bound is denoted by  $\underline{A}${, and} its upper bound is denoted by $\overline{A}$, i.e. $\underline{[a,b]}=a$ and $\overline{[a,b]}=b$ for every $[a,b]\in \con{IR}$.

Since intervals are {sets, the} inclusion determines an order. Observe that the inclusion order for intervals can be determined exclusively on their extremes as follows
$$[a,b]\subseteq  [c,d] \Leftrightarrow c\leq a\wedge b\leq d.$$
Auxiliarly, we also define the following  strict order on $\con{IR}$:
$$[a,b]\Subset [c,d] \Leftrightarrow c< a\wedge b< d.$$
 Since, $[3,4]\subseteq [3,5]$ but $[3,4]\not\Subset [3,5]$ then $\subseteq\neq \Subset$.

In \cite{miranker1981computer} consider the following order for $\con{IR}$:
$$[a,b]\leq_{KM} [c,d]\Leftrightarrow a\leq c\wedge b\leq d.$$
This order is not linear and{, in some} situations a linear order is fundamental (see for example \cite{deng2012fuzzy}). Of course, there are infinitely many linear orders on $\con{IR}$. This motived \cite{bustince2013generation}, in the context of interval-valued fuzzy sets, i.e. in $L([0,1])=\{[a,b]\in  \con{IR}: 0\leq a\leq b\leq 1\}$,   to introduce the notion of admissible linear orders. For Bustince, in \cite{bustince2013generation}, an order only is admissible if it refines the usual  order on $L([0,1])$. But, it clear that this notion can be {adapted} in a straightforward way for $\con{IR}$: 

\begin{definicion}\label{def-ordenadmisible1}
A relation $\preceq$ on $\con{IR}$ is called an admisible order, if 
\begin{enumerate}[labelindent=\parindent, leftmargin=*,label=]
\item[(i)] $\preceq$ is a linear order on $\con{I}\con{R}$;
\item[(ii)] for all $A$, $B$ on $\con{IR}$, $A\preceq B$ whenever $A\leq_{KM} B$.
\end{enumerate}
\end{definicion}

\begin{ejemplo}\label{Exemple2.1}
Admissible orders on $\con{IR}$:
 \begin{enumerate}
  \item The Lexical 1: $[a,b]\preceq_{Lex1} [c,d]\Leftrightarrow a< c\vee (a=c\wedge b\leq d)$;
  
  \item The Lexical 2: $[a,b]\preceq_{Lex2} [c,d]\Leftrightarrow b< d\vee (b=d\wedge a\leq c)$;
  
  \item Xu-Yager (adapted from \cite{xu2006some}): $$[a,b]\preceq_{XY} [c,d]\Leftrightarrow a+b< c+d\vee (a+b=c+d\wedge b-a\leq d-c);$$
  
  \item Twice Xu-Yager (adapted from \cite[Ex. 4]{Santana2020}): 
  \begin{multline*}
[a,b]\preceq_{2XY} [c,d]\Leftrightarrow a+3b< c+ 3d\\\vee (a+3b=c+3d\wedge b-a\leq d-c).
 \end{multline*}
 \end{enumerate}
\end{ejemplo}

\subsection{Fuzzy sets}

The following definitions can be found in \cite{Klir1995,Bector2005} and in most of the introductory books on fuzzy sets theory. In all this section $X$ will be a non-empty reference set with generic elements denoted by $x$.

\begin{definicion}\label{definicionsoporte}  A fuzzy set $A$ on $X$ is a function $\operador{A}{X}{[0, 1]}$.
In addition, 
 \begin{enumerate}
 \item[(i)] The support of $A$,  is the set $supp(A)=\{x\in X\talque A (x) > 0\}$;
  \item[(ii)] The kernel of $A$,  is the  set $\ker(A)=\{x\in X\talque A (x) =1\}$;
  \item[(iii)] Given  $\alpha\in~]0,1]$, the $\alpha$-cuts set of  $A$ is the set $\cortea{A} =\{x\in X \talque A(x)\geq \alpha\}$;
  \item[(iv)] The height of $A$ is  $h(A) = \Sup{x\in X} A (x)$.
 \end{enumerate}
\end{definicion}

If $h(A) = 1$, then the fuzzy set $A$ is called of normal fuzzy set, otherwise, i.e. if $h(A)<1$, it is called subnormal. Clearly, in a finite set $X$, we have, $A$ is normal if and only if  $\ker(A)\neq \emptyset$ and if $A$ is  subnormal and  $supp(A)\neq \emptyset$ then it can be normalized, by the new fuzzy set $A^*$ on $X$ where $A^\ast(x)=\frac{A(x)}{h(A)}$, for each $x\in X$.

\begin{definicion}
A fuzzy set $A$ on $\con{R}$ is said  to be a convex fuzzy set if its $\alpha$-cuts are (crisp) convex sets, i.e. for all $\alpha,t\in~]0, 1]$ and  $x,y\in \cortea{A}$, $tx+(1-t)y\in\cortea{A}$.
\end{definicion}

\begin{teorema}\cite[Theorem 1.1]{Klir1995}\label{teoconvexidade}
A fuzzy set $A$ on $\con{R}$ is  convex  iff for all $x_1$, $x_2\in\con{R}$ and $\lambda\in[0,1]$
$$A(\lambda x_1+(1-\lambda)x_2)\geq \min\left\{A(x_1),A(x_2)\right\},$$
where $\min$ denotes the minimum operator.
\end{teorema}

\subsection{Fuzzy numbers}

There are several different definitions of fuzzy numbers in the literature, for example \cite{valvis2009new,Klir1995,Bector2005,hanss2005applied,lee2004first,buckley2002introduction,gomes2015fuzzy} {Most} of them vary in the kind of continuity required for the membership function. For example, in \cite{valvis2009new,gomes2015fuzzy} is considered upper semi-continuity whereas in \cite{hanss2005applied,lee2004first} is required piecewise continuity and in \cite{Klir1995,Bector2005,buckley2002introduction} no continuity constraint is required. Another difference {can be that some require} that the kernel of the fuzzy number be a singleton{, another one is that it should be non-empty}. Here we adopted the given in \cite{Klir1995}.

\begin{definicion} \label{def-FN} A fuzzy set $A$ on $\con{R}$ is called a fuzzy number if it satisfies the following conditions
\begin{enumerate}[labelindent=\parindent, leftmargin=*,label=]
\item[(i)] $A$ is normal;
\item[(ii)] $\cortea{A}$ is a closed interval for every $\alpha\in ~]0, 1]$;
\item[(iii)] the support of $A$ is bounded.
\end{enumerate}
\end{definicion}
Finally, $\NFR$ will denote the set of all fuzzy numbers.

We note that Definition \ref{def-FN}  is equivalent to appear in \cite[p.44]{Bector2005}.

\begin{observacion}\label{sobreelsoporteconteoremadecaracterizaciondefuzzy}
Since the support  of a fuzzy number of $A$ is bounded, there exist  $\omega_1$, $\omega_2$ in $\con{R}$, s.t. $supp(A)=\{x\in\con{R} \talque \omega_1<x<\omega_2\}$. In  addition, we will use the notation $supp(A)^-$ and $supp(A)^+$ for such bounds, i.e. for $\omega_1$ and $\omega_2$, respectively. Analogously, since the kernel  of $A$ is convex then there exists a closed interval $[a,b]$, s.t. $\ker(A)=[a,b]$. {{Besides, we will use notation $\ker(A)^-$ and $\ker(A)^+$ for such bounds, i.e. for $a$ and $b$, respectively.} }
\end{observacion}
The next theorem gives a full characterization of fuzzy numbers.

\begin{teorema}\cite[Theorem  4.1]{Klir1995} \label{teoKlir19954-1}
Let $A$ be a fuzzy set on $\mathds{R}$. Then, $A\in \mathcal{F}(\con{R})$ if and only if there exist a closed interval $[a, b]\neq \emptyset$, a function $l$ from $[-\infty, a]$ to $[0, 1]$ which is right-continuous, increasing  and $l(x)=0$ for each $x\in~]-\infty,supp(A)^-]$, and  a function $r$ from $[b,+\infty]$ to $[0, 1]$ which is left-continuous, decreasing and $r(x)=0$ for each $x\in [supp(A)^+,+\infty[$, such that
\begin{equation}\label{eqteoKlir19954-1}
A(x)=\left\{\begin{array}{ll}
1, & \textrm{ if }x \in [a, b],\\
l(x), & \textrm{ if } x \in~ ]-\infty, a[,\\
r(x), & \textrm{ if }x \in~]b,+\infty[.\\
\end{array}\right.
\end{equation}
\end{teorema}

\begin{corolario}
For each interval $[a,b]\in \con{IR}$ their characteristic function 
$\widetilde{[a,b]}:\con{R}\rightarrow [0,1]$ defined by

\[\widetilde{[a,b]}(x)=\left\{\begin{array}{ll}
                               1, & \mbox{ if }x\in [a,b], \\
                               0, & \mbox{ if }x\not\in [a,b]
                              \end{array}
                              \right. \]
is a fuzzy number. 
\end{corolario}

So, in some sense, we can think that fuzzy numbers generalize the set of closed intervals of real numbers, i.e. that $\con{IR}\subseteq  \mathcal{F}(\con{R})$ and therefore  $\con{R}\subseteq  \mathcal{F}(\con{R})$ too, once degenerated intervals can be seen as real numbers and instead of {writing} $\widetilde{[a,a]}$ we just use $\widetilde{a}$. 
A fuzzy number  $\widetilde{a}$ is called of crisp number or fuzzy singleton in \cite{hanss2005applied}. A fuzzy number $A$  is called a triangular fuzzy number whenever $ker(A)=[a,a]$, $l(x)=\frac{x-supp(A)^-}{a-supp(A)^-}$ and $r(x)=\frac{supp(A)^+ - x}{supp(A)^+ -a}$, and is denoted by the triple $(supp(A)^-,a,supp(A)^+)$.

\begin{observacion} \label{rem-uppercontinuidad}
 From \cite[Remark 3.3.2.]{Bector2005} we have that each fuzzy number $A$ is an upper semi-continuous function and therefore, the definition given  \cite{gomes2015fuzzy,valvis2009new} is equivalent to the Definition \ref{def-FN}.
\end{observacion}

In the proof of Theorem \ref{teoKlir19954-1}, Klir and Yuan, provide a characterization of the $\alpha$-cuts of fuzzy numbers based on  the functions $l^\ast: [0,1]\rightarrow ~]-\infty,a[$ and $r^\ast:[0,1]\rightarrow~ ]b,\infty[$  defined by 
\begin{equation*}\label{eq-last-rast} 
 l^\ast(\alpha)=\inf\left\{x\in~ ]-\infty,a[~:l(x)\geq \alpha\right\}
\end{equation*}
and  
\begin{equation*}\label{eq-last-rast} 
r^\ast(\alpha)=\sup\left\{x\in~]b,\infty[~:r(x)\geq \alpha\right\}
\end{equation*} 
with the convention that {if $\left\{x\in~ ]-\infty,a[~:l(x)\geq \alpha\right\}=\left\{x\in~]b,\infty[~:r(x)\geq \alpha\right\}=\emptyset$ then $l^\ast(\alpha)=a$ and  $r^\ast(\alpha) = b$,} and
given by
\begin{equation}\label{parametrizacionalphacorte}
 \cortea{A}= \begin{cases}\begin{array}{ll}
  [l^{\ast}(\alpha),r^{\ast}(\alpha)],& \textrm{ if }\alpha\in [0,1[,\\
 [a,b],&\textrm{if }\alpha=1.\end{array}\end{cases}
 \end{equation}

\begin{ejemplo} Consider the fuzzy number $A$ (see Figure \ref{ejemplo-alfa-cortes}) given by:
$$A(x)=\begin{cases} \begin{array}{ll}
1,& \mbox{ if }x\in[3,4],\\ l(x),& \mbox{ if }x\in~ ]-\infty,3[,\\ r(x), & \mbox{ if }x\in~]4,+\infty[,
\end{array}\end{cases}$$
where $l$ and $r$ are:
$$l(x)=\begin{cases}\begin{array}{ll}\frac{x+1}{4}, &\mbox{ if }2\leq x<3,\\
\frac{1}{2}, & \mbox{ if }1\leq x< 2,\\ 0, &\mbox{ if } x<1,
\end{array}\end{cases}$$
and
$$r(x)=\begin{cases} \begin{array}{ll}
 \frac{20-3x}{8}, & \mbox{ if }4<x\leq 5\\ \frac{6-x}{3}, & \mbox{ if }5<x\leq 6,\\ 0, & \mbox{ if }x>6,\end{array}
\end{cases}$$
\begin{figure}[h]
\centering
\includegraphics[width=0.4\textwidth]{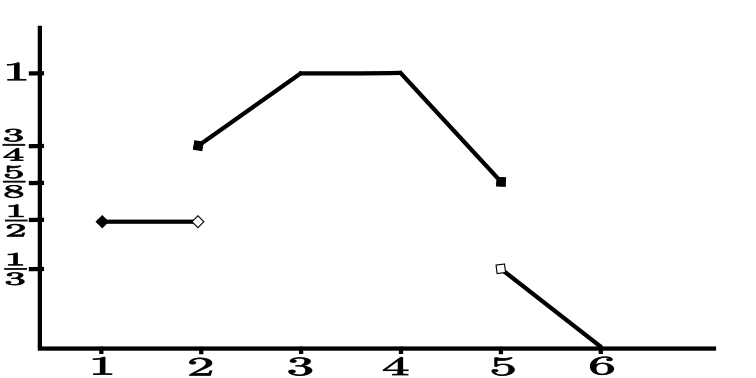}
\caption{Fuzzy Number $A$.}
\label{ejemplo-alfa-cortes}
\end{figure}
verify the conditions of the Theorem \ref{teoKlir19954-1}. Let's calculate the $\alpha$-cuts of $A$, starting with $l$, i.e.:

\begin{equation}\label{lasteristo}
l^*(\alpha)=\left\{\begin{array}{ll}
3, &\textrm{ if } \alpha=1,\\
1-4\alpha, &\textrm{ if } \alpha\in\left[\frac{3}{4},1\right[,\\
2, &\textrm{ if } \alpha\in\left]\frac{1}{2},\frac{3}{4}\right[,\\
1, &\textrm{ if } \alpha\in \left[0,\frac{1}{2}\right].
\end{array}\right.
\end{equation} 
To continue, with $r$ it's analog, we have
\begin{equation}\label{rasteristo}
r^*(\alpha)=\left\{\begin{array}{ll}
4,&\textrm{ if } \alpha=1,\\
\frac{20-8\alpha}{3},&\textrm{ if } \alpha\in\left[\frac{5}{8},1\right[,\\
5,&\textrm{ if } \alpha\in\left[\frac{1}{3},\frac{5}{8}\right],\\
6-3\alpha, &\textrm{ if } \alpha\in\left[0,\frac{1}{3}\right[.
\end{array}\right.
\end{equation} 
Therefore, from Eq. \eqref{parametrizacionalphacorte}, we express the $\alpha$-cuts of $A$ given by:
\begin{align*}
\cortea{A}&=\left\{\begin{array}{ll}
[l^\ast(\alpha),r^\ast(\alpha)], & \textrm{ if }\alpha\in [0,1[,\\
\left[a,b\right], & \textrm{ if }\alpha=1
    \end{array}\right.\\
 &=\left\{\begin{array}{ll}
\left[4\alpha-1,\frac{20-8\alpha}{3}\right], & \textrm{ if }\alpha\in\left[\frac{3}{4},1\right[,\\
\left[2,\frac{20-8\alpha}{3}\right], & \textrm{ if }\alpha\in\left[\frac{5}{8},\frac{3}{4}\right[,\\
\left[2,5\right], & \textrm{ if }\alpha\in\left]\frac{1}{2},\frac{5}{8}\right[,\\
\left[1,5\right], & \textrm{ if }\alpha\in\left[\frac{1}{3},\frac{1}{2}\right],\\
\left[1,6-3\alpha\right], & \textrm{ if }\alpha\in\left[0,\frac{1}{3}\right[,\\
\left[3,4\right], & \textrm{ if }\alpha=1.\\
\end{array}\right.
\end{align*}
\end{ejemplo}
 
            \begin{proposicion}\cite[p. 109-110]{Klir1995} Let $A,B\in \mathcal{F}(\con{R})$ then the fuzzy sets $A\wedge B$ and $A\vee B$ defined by
$$A\wedge B(x)=\Sup{x=\min\{y,z\}} \min\{A(y),B(z)\}$$
and
$$A\vee B(x)=\Sup{x=\max\{y,z\}} \min\{A(y),B(z)\},$$
for each $x\in \con{R}$ is a fuzzy number. In addition, $\langle\mathcal{F}(\con{R}),\wedge,\vee\rangle$ is a distributive lattice.
\end{proposicion}
            
The {arithmetic operations} on fuzzy numbers are defined based on the Zadeh extension principle.
 Let $A$ and $B$ be two fuzzy numbers and $\star\in \{+,-,\cdot, \div\}$ 
  and the fuzzy set $A\star B$ defined for $z$ in  $\con{R}$    
  as $A\star B(z)=\sup\limits_{z=x\star y} \min\{A(x),B(y)\}$.   
  Then $A\star B$ is also a fuzzy number \cite[Theorem 4.2]{Klir1995}. An alternative form to define the arithmetic operations on $\mathcal{F}(\con{R})$ is based on the Klir and Yuan decompositional theorem  \cite[Theorem 2.5]{Klir1995} which {proves} that each fuzzy set can be recovered from its $\alpha$-cuts. Thereby, $A\star B$ is the fuzzy set whose $\alpha$-cuts are 
  \begin{equation}\label{eq-alpha-cuts-A*B}
   \cortea{A\star B}=\cortea{A}{\diamond}\cortea{B},
  \end{equation}
where $\diamond\in \{+,-,\cdot, \div\}$ is the respective arithmetic operation on $\IR$ for all $\alpha\in~]0,1]$.
  In addition, when $A$ and $B$ are triangular fuzzy numbers then $A+B$ and $A-B$ are also triangular fuzzy numbers, but $A\cdot B$ and $A\div B$ {can not be} triangular fuzzy numbers \cite[Section 3.5]{Bector2005}. Indeed, $A=(a_1,a_2,a_3)$ and $B=(b_1,b_2,b_3)$ then $A+B=(a_1+b_1,a_2+b_2,a_3+b_3)$ and $A-B=(a_1-b_1,a_2-b_2,a_3-b_3)$. Notice also that for each $r\in \con{R}$ and a triangular fuzzy number $A=(a_1,a_2,a_3)$ we have that {$\crispy{r}\cdot A=(a_1\cdot r, a_2 \cdot r, a_3\cdot r)$}.

\section{Order on fuzzy numbers}

The following partial order in $\mathcal{F}(\con{R})$  was {proposed} by Zadeh in \cite{zadeh1965fuzzy}.
\begin{definicion}\label{orderzadeh}
Let $A$ and $B$ be two fuzzy numbers.
$$A\leq_Z B\Longleftrightarrow A (x ) \leq B(x) \textrm{ for all } x\in\con{R}.$$
\end{definicion}
The Figure-\ref{f:incomparables1} {shows} a case where $A\leq_Z B$.

The Zadeh's order can be characterized in terms of the inclusion order on their $\alpha$-cuts.

\begin{proposicion}\cite[Theorem 2.3-(viii)]{Klir1995}
\label{equivalencia-zadeh-moore}
Let $A$ and $B\in \NFR$. Then $A\leq_Z B$ if and only if $\forall\alpha\in~]0,1]~~ \cortea{A}\subseteq \cortea{B}$.
\end{proposicion}

The problem with this order is that it {does not generalize} the usual order on the real numbers. In fact,  given $x,y\in \con{R}$ such that $x< y$, we have that {$\crispy{x}\not\leq_Z \crispy{y}$}. 

Klir and Yuan in \cite{Klir1995} proposed the following partial order on  $\mathcal{F}(\con{R})$: 

Let $A$ and $B\in \NFR$. Then 
$$A\leq_{KY} B \Longleftrightarrow A\wedge B=A.$$

\begin{proposicion} \cite[p. 114]{Klir1995}\label{teoordenalphacorte}
 Given fuzzy number $A$ and $B$, the following assertions are equivalents
 \begin{enumerate}
  \item $A\leq_{KY} B$;
  \item $A\vee B=B$;
  \item $\cortea{A} \leq_{KM} \cortea{B}$ for each $\alpha\in~]0,1]$.
 \end{enumerate}
\end{proposicion}

Observe that the Klir-Yuan partial order{, when restricted to intervals, corresponds} to the  Kulisch-Miranker order  and when {restricted} to real numbers it {agrees} with the usual order.
The problem {is that} there are {pairs} of fuzzy numbers which are non-comparable under this order. The Figures \ref{f:incomparables1} and \ref{f:incomparables2} present the two generic cases of pairs of fuzzy numbers which are non-comparable by the partial order $\leq_{KY}$ and therefore, $\leq_{KY}$ is not a linear order.
\begin{figure}[h]
 \centering
  \subfloat[]{
   \label{f:incomparables1}
    \includegraphics[width=0.45\textwidth]{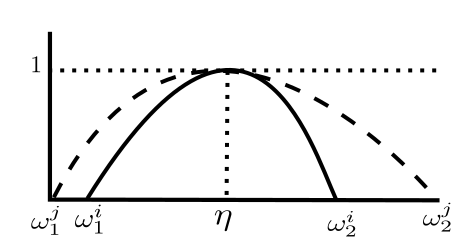}}\\
      \subfloat[]{
   \label{f:incomparables2}
    \includegraphics[width=0.45\textwidth]{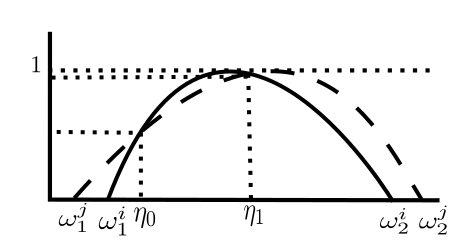}}
           \label{f:numeros}
           \caption{General cases of pairs of non-comparable fuzzy number with respect the Klir-Yuan order.}
 \end{figure}
 
 From the above observation, we get the following characterization of the non-comparable fuzzy numbers for this order.
 
 \begin{corolario}
Let  $A$ and $B$ be fuzzy numbers. $A$ and $B$ are non-comparable in the order $\leq_{KY}$, denoted by $A\parallel_{KY} B$ if, and only if, 
\begin{enumerate}
 \item there exists $\alpha\in~]0,1]$ such that $\cortea{A} \Subset\cortea{B}$ or $\cortea{B} \Subset \cortea{A}$; or
 \item there exist $\alpha,\beta\in~]0,1]$ such that $\cortea{A} <_{KM} \cortea{B}$ and $\nicefrac{B}{\beta} <_{KM} \nicefrac{A}{\beta}$.
\end{enumerate}
 \end{corolario}

 \begin{observacion}
  \label{rem-leqLK-unbounded} Since for each positive real number $r$ and $A\in \NFR$ we have that {$\cortea{A_{-r}}<_{KM}\cortea{ A }<_{KM}\cortea{ A_{+r}}$, for all $\alpha\in~]0,1]$ and }where $A_{-r}(x)=A(x+r)$ and $A_{+r}(x)=A(x-r)$, the distributive lattice $\langle\mathcal{F}(\con{R}){,\wedge,\vee},\leq_{KY}\rangle$ is not bounded and therefore not a complete lattice. {Also, $\langle\mathcal{F}(\con{R}){,+,\cdot},\leq_{KY}\rangle$ is a subdistributive lattice (see \cite[Pag. 104, point 4]{Klir1995}) whinch is not bounded and therefore it is not a complete lattice. }
 \end{observacion}

\subsubsection{Wang-Wang order}

Wei Wang and Zhenyuan in \cite{wang2014total} propose a total order for the set of fuzzy numbers based  on $\alpha$-cuts from a special type of sequence in $[0,1]$.

\begin{definicion}(\cite{wang2014total})\label{defdensidadS}
Let $S=(\alpha_i)_{i\in\con{Z}^+}$ be a sequence in $]0,1]$, where $\con{Z}^+$ is the set of positive integers. Then $S$  is upper dense  if, for every point $x\in~ ]0, 1]$ and any $\varepsilon > 0$, there exists $i\in \con{Z}^+$ such that
$\alpha_i\in [x, x + \varepsilon[$. 
\end{definicion}

\begin{observacion} \label{obs-S}
If  $S=(\alpha_i)_{i\in\con{Z}^+}$   is an upper dense sequence in $]0,1]$ then
\begin{enumerate}
 \item $\inf S=0$ and $\sup S=1$;
 \item for all $n\in \con{Z}^+$ {there are $i<j\in \con{Z}^+$} such that {$n\leq i$}, $\alpha_i < \alpha_{i+1}$ and $\alpha_{j+1}< \alpha_j$;
 \item for all $n\in \con{Z}^+$ the sequence $S_n=(\alpha'_{i})_{i\in\con{Z}^+}$ with $\alpha'_i=\alpha_{n+i}$ is also an upper dense sequence in $]0,1]$;
 \item for all $\alpha\in~]0,1]$, the sequence $S_\alpha=(\alpha'_{i})_{i\in\con{Z}^+}$ with $\alpha'_i=\alpha_{i-1}$ for each $i\geq 2$ and $\alpha'_1=\alpha$  is also an upper dense sequence in $]0,1]$;
 \item in Definition \ref{defdensidadS}, when say $\alpha_i\in [x, x + \varepsilon[$, we have that $x + \varepsilon$ not need belong to $[0,1]$ but $\alpha_i\in~]0,1]$ and therefore $\alpha_i\in [x, x + \varepsilon[~\cap~ [x,1]$.
\end{enumerate}
\end{observacion}

\begin{observacion}
When you take $x=1$ in the Definition \ref{defdensidadS} then there is a $i\in\con{Z}^+$ such that $\alpha_i\in [1,1+\epsilon]$. Then in all upper dense sequence $S=
(\alpha_i)_{i\in\con{Z}^+}$ exists $k$ such that $\alpha_k=1$. But by the Remark \ref{obs-S}-(3), the  sequence $S_{k+1}$ is upper dense and therefore $\alpha'_j=1$ for some $j\in\con{Z}^+$. So, each upper dense sequence has infinite copies of 1.
\end{observacion}

Examples of these sequences can be found in  \cite[Example 1 and 2]{wang2014total}.

\begin{definicion}[\cite{wang2014total}]\label{def3.3ww}
Let $A$ be fuzzy numbers. For a given upper dense sequence $S=(\alpha_i)_{i \in \con{Z}^+}$ in $]0,1]$, we define $\operador{c_i}{\NFR}{\con{R}}$ given by
$$c_i(A)=\begin{cases}\begin{array}{ll}
r^{\ast}(\alpha_{\frac{i}{2}})-l^{\ast}(\alpha_{\frac{i}{2}}),&\textrm{ if }i\textrm{ is even},\\
l^{\ast}(\alpha_{\frac{i+1}{2}})+r^{\ast}(\alpha_{\frac{i+1}{2}}),&\textrm{ if }i\textrm{ is odd}.
\end{array}\end{cases}$$     
\end{definicion}

\begin{definicion}[\cite{wang2014total}]\label{wang2014total-def}
Let $A$ and $B$ be two fuzzy numbers and an upper dense sequence $S=(\alpha_i)_{i \in \con{Z}^+}$ in $]0,1]$. We say that $A<_{WW}^SB$ 
when there exists a positive integer $n_0$ such that $c_{n_0}(A)<c_{n_0}(B)$  and $c_i(A)=c_i(B)$ for all positive integers $i<n_0$. We say that $A\leq_{WW}^S B$ if, and only if, $A<_{WW}^S B$ or $A=B$.
\end{definicion}

As it is well know, any fuzzy set $A$ can be fully identified with its $\alpha$-cuts in the following sense:
$$A(x)=\sup_{\alpha\in (0,1]} \alpha\cdot \chi_{\cortea{A}}(x),$$
where $\chi_{\cortea{A}}$ is the characteristic function of the interval $\cortea{A}$. This is called decomposition theorem \cite[Theorems 2.5]{Klir1995}. There are some  variants of this theorem  such as \cite[Theorem 3]{wang2014total} and  \cite[Theorems 2.6 and 2.7]{Klir1995}. In particular, Wang and Wang variant proves that  any fuzzy number is recovered from just a countably subset of their $\alpha$-cuts.

\begin{teorema}\cite[Theorem 3]{wang2014total} \label{teo-decIV}
 Let $A$ be a fuzzy number and $S=(\alpha_i)_{i\in\con{Z}^+}$ be an upper dense sequence  in $]0,1]$. Then 
 $$A(x)=\sup_{i\in \con{Z}^+} \alpha\cdot \chi_{\corteai{A}{i}}(x).$$
\end{teorema}

\begin{corolario}\label{coro-teo-decIV}
Let $A$ and $B$ be two fuzzy numbers and $S=(\alpha_i)_{i\in\con{Z}^+}$ be an upper dense sequence  in $]0,1]$. $A=B$ if, and only if, $\corteai{A}{i}=\corteai{B}{i}$, for all $i\in \con{Z}^+$.
\end{corolario}
\begin{proof}
 Straightforward from Theorem \ref{teo-decIV}.
\end{proof}

\begin{teorema}\cite{wang2014total}\label{ordenwang}
Let $S$ be an upper dense sequence in $]0,1]$. Then $\leq_{WW}^S$ is a linear order on $\NFR$.
\end{teorema}

\begin{observacion}
There exist several pairs of upper dense sequences $S_1$ and $S_2$ in $]0,1]$  which determine distinct linear orders, i.e.  $\leq_{WW}^{S_1}\neq  \leq_{WW}^{S_2}$. Therefore, we are dealing with a family of linear orders (see Example 3 \cite{wang2014total}).
\end{observacion}

\section{Admissible orders on fuzzy numbers}

\begin{definicion}\label{def-ordenadmisible}
Let $\leqq$ and $\preceq$ be two orders on $\NFR$. The order $\preceq$ is called an admissible order w.r.t. $\prt{\NFR,\leqq}$, if 
\begin{enumerate}[labelindent=\parindent, leftmargin=*,label=]
\item[(i)] $\preceq$ is a linear order on $\NFR$;
\item[(ii)] for all $A$, $B$ in $\NFR$, $A\preceq B$ whenever $A\leqq B$.
\end{enumerate}
\end{definicion}

Thus, an order $\preceq$ on $\NFR$ is admissible for $\prt{\NFR,\leqq}$, if it is linear and refines the order $\leqq$. In particular, when the order $\leqq$ is $\leq_{KY}$ we will call $\preceq$ just of admissible order on $\NFR$. Furthermore, if $\leqq$ is a linear order then $\preceq$ and $\leqq$ are the same.

\begin{proposicion}
Let $\preceq$ be an admissible order on $\NFR$. Then, there are not greatest or smallest elements in $\NFR$.
\end{proposicion}
\begin{proof} 
Straightforward from Definition \ref{def-ordenadmisible} and Remark \ref{rem-leqLK-unbounded}.
\end{proof}

\begin{definicion}
Let $A,B\in\NFR$ be such that $A\neq B$ and $S=(\alpha_i)_{i\in \con{Z}^+}$ be an  upper dense sequence in $]0,1]$. Then  define $m(A,B)$  by
$$m(A,B)= \left\{\begin{array}{ll}
                     \min\{i\in \con{Z}^+: \corteai{A}{i}\neq \corteai{B}{i}\}, & \mbox{ if $A\neq B$}, \\
                     0, & \mbox{otherwise}.
                 \end{array}
                 \right.
                    $$
\end{definicion}
Observe that $m(A,B)=m(B,A)$ and that, by Corollary  \ref{coro-teo-decIV}, $m(A,B)$ is well defined.

\begin{proposicion}\label{lemanew}
Let  $S=(\alpha_i)_{i\in\con{Z}}$ be an upper dense sequence in $]0,1]$ and $A,B\in\NFR$.  $\{\alpha_i\in S:\corteai{A}{i}\neq \corteai{B}{i}\}\neq\{1\}$.
\end{proposicion}

\begin{proof} If $A=B$ then  $\{\alpha_i\in S:\corteai{A}{i}\neq \corteai{B}{i}\}=\emptyset \neq\{1\}$. If $A\neq B$ then, by Corollary \ref{coro-teo-decIV}, we have that $\corteai{A}{i}\neq \corteai{B}{i}$ for some $i\in \con{Z}^+$. Suppose that $\alpha_i=1$ then we have four cases:  $\ker(A)^-< \ker(B)^-$, $\ker(A)^- > \ker(B)^-$, $\ker(A)^+< \ker(B)^+$ or $\ker(A)^+> \ker(B)^+$. In the first case, take $x\in \ker(A)^-$. Then $A(x)=1$ and $B(x)<1$. So, for $\alpha=\frac{B(x)+1}{2}$, we have that $x\in \cortea{A}$ and $x\not\in \cortea{B}$, i.e.
$\cortea{A} \neq \cortea{B}$. So, since $S$ is an upper sequence in $]0,1]$, there exists $j\in \con{Z}^+$ such that $\alpha \leq \alpha_j < 1$. Since,  $x\in \corteai{A}{j}$ and $x\not\in\corteai{B}{j}$, then  $\corteai{A}{j}\neq \corteai{B}{j}$ and therefore, $\{\alpha_i\in S:\corteai{A}{i}\neq \corteai{B}{i}\}\neq\{1\}$.

The other three cases are similarly proved. 
\end{proof}

\begin{definicion}
\label{adinter}
Let   $S=(\alpha_i)_{i\in\con{Z}}$ be an upper dense sequence in $]0,1]$, $\preceq$ be an order on $\con{IR}$,  and $A,B\in\NFR$. Then,
$$A\unlhd^S B\Longleftrightarrow  A=B\mbox{ or } \corteai{A}{m(A,B)} \prec \corteai{B}{m(A,B)}.$$ 
\end{definicion}
 
 Observe that, taking as convention that $\corteai{A}{0}=[supp(A)^-,supp(A)^+]$  then $A\unlhd^S B\Longleftrightarrow   \corteai{A}{m(A,B)} \preceq \corteai{B}{m(A,B)}$.

\begin{teorema}\label{ZBM}
Let $\preceq$ be an admissible order on $\con{IR}$ and $S=(\alpha_i)_{i\in\con{Z}^+}$ be an upper dense sequence in $]0,1]$. The relation $\unlhd^S$ is an admissible order on $\NFR$.
\end{teorema}
\begin{proof}
Let $S=(\alpha_i)_{i\in\con{Z}}$  be an upper dense sequence in $]0,1]$. 

\noindent{\underline{Reflexivity}:} Straightforward from  Definition \ref{adinter}.

\noindent{\underline{Antisymmetry}:}  Let $A$, $B$ fuzzy numbers such that $A\unlhd^S B$ and $B\unlhd^S A$. Suppose that $A\neq B$ then, from Corollary \ref{coro-teo-decIV}, $\{i\in \con{Z}^+: \corteai{A}{i}\neq \corteai{B}{i}\}\neq \emptyset$. Let $m=m(A,B)=\min \{i\in \con{Z}^+: \corteai{A}{i}\neq \corteai{B}{i}\}=m(B,A)$, then, from the former $\corteai{A}{m}\preceq\corteai{B}{m}$ and $\corteai{B}{m}\preceq \corteai{A}{m}$. So, because $\preceq$ is an order, $\corteai{A}{m}= \corteai{B}{m}$, which is a contradiction.
Therefore, $A=B$.

\noindent{\underline{Transitivity}:} Let $A$, $B$, and $C$ be three fuzzy numbers such that $A\unlhd^S B$ and $B\unlhd^S C$. If $A=B$ or $B=C$ then trivially $A\unlhd^S C$ and if $A=C$ then, by antisymmetry, $A=B=C$. If  $A\neq B$ and $B\neq C$ then $A\lhd^S B$ and $B\lhd^S C$. So, $\corteai{A}{k}\prec\corteai{B}{k}$ and $\corteai{B}{m}\prec\corteai{C}{m}$, where $k=m(A,B)$ and $m=m(B,C)$. 
If $k \leq m$ then $\corteai{B}{k}\preceq \corteai{C}{k}$ and since $\corteai{A}{k}\prec \corteai{B}{k}$ then $\corteai{A}{k}\prec \corteai{C}{k}$. In addition, if $j<k$ then  $\corteai{A}{j}= \corteai{B}{j}$ and $\corteai{B}{j}= \corteai{C}{j}$, and therefore $\corteai{A}{j}= \corteai{C}{j}$. Therefore, $A\unlhd^S C$. Analogously, if $m< k$ we prove that $A\unlhd^S C$. This means that the relation is transitive.

\noindent{\underline{Totallity}:}  Let $A$ and $B$ be two fuzzy numbers such that $A\neq B$. Then, $\corteai{A}{m}\neq \corteai{B}{m}$ for $m=m(A,B)$. Thus,
 because $\preceq$ is total $\corteai{A}{m}\prec\corteai{B}{m}$ or $\corteai{B}{m}\prec \corteai{A}{m}$. 
 Therefore $A\unlhd^S B$ or $B\unlhd^S A$ for all $A,B\in\NFR$. 

\noindent{\underline{Refinement}:} Let $A$ and $B$ be two fuzzy numbers such that  $A\leq_{KY} B$. If $A=B$ then, since $\unlhd^S$ is reflexive, we have that $A\unlhd^S B$. If $A<_{KY} B$ then, by Proposition \ref{teoordenalphacorte}, $\cortea{A}\leq_{KM} \cortea{B}$ for each $\alpha\in (0,1]$ and therefore, for $m=m(A,B)$, we have that $\corteai{A}{m}<_{KM} \corteai{B}{m}$ and $\corteai{A}{k}=\corteai{B}{k}$ for each $k\leq m$. So, as $\preceq$ is admissible order on $\con{IR}$, $\corteai{A}{m}\prec \corteai{B}{m}$. Thereby,  $A\unlhd^S B$.

\noindent Therefore, $\unlhd_S$ is an admissible order.
\end{proof}

We will denote as $\unlhd_{Lex1}^S$, $\unlhd_{Lex2}^S$, $\unlhd_{XY}^S$ and $\unlhd_{2XY}^S$ the admissible orders on $\NFR$ generated by the admissible orders $\preceq_{Lex1}$, $\preceq_{Lex2}$, $\preceq_{XY}$ and $\preceq_{2XY}$ and an upper dense sequence $S=(\alpha_i)_{i\in\con{Z}^+}$  in $]0,1]$, respectively, according to Theorem \ref{ZBM}.

  \begin{proposicion} \label{pro-leq-WW-leq-XY}
  Let $S=(\alpha_i)_{i\in\con{Z}^+}$ be an upper dense sequence in $]0,1]$. Then 
   $\unlhd_{XY}^S=\leq_{WW}^S$.
  \end{proposicion}

  \begin{proof}
  Let $A,B\in \NFR$. If $A<_{WW}^S B$ then there exists $n_0\in \con{Z}^+$ such that $c_{n_0}(A) < c_{n_0}(B)$ and $c_i(A)=c_i(B)$ for each $i<n_0$. We consider two cases, namely:
 
\ITEM{Case 1.} If $n_0$ is even then taking $m=\frac{n_0}{2}$  we have that $\rcorte{A}{m}-\lcorte{A}{m}< \rcorte{B}{m}-\lcorte{B}{m}$, $\rcorte{A}{m}+\lcorte{A}{m}=\rcorte{B}{m}+\lcorte{B}{m}$,  $\rcorte{A}{m-i}-\lcorte{A}{m-i}=\rcorte{B}{m-i}-\lcorte{B}{m-i}$ and $\rcorte{A}{m-i}+\lcorte{A}{m-i}=\rcorte{B}{m-i}+\lcorte{B}{m-i}$  for each $i< m$. Hence, $\rcorte{A}{m}-\lcorte{A}{m}< \rcorte{B}{m}-\lcorte{B}{m}$, ${\rcorte{A}{m}}+{\lcorte{A}{m}}={\rcorte{B}{m}}+{\lcorte{B}{m}}$, ${\rcorte{A}{i}}+{\lcorte{A}{i}}= {\rcorte{B}{i}}+{\lcorte{B}{i}}$ and ${\rcorte{A}{i}}-{\lcorte{A}{i}}= {\rcorte{B}{i}}-{\lcorte{B}{i}}$ for all $i<m$.

\ITEM{Case 2.} If $n_0$ is odd then taking $m=\frac{n_0+1}{2}$  we have that $r^*_A(\alpha_{m})+l^*_A(\alpha_{m}) < r^*_B(\alpha_{m})+l^*_B(\alpha_{m})$,  $r^*_A(\alpha_{m-i})-l^*_A(\alpha_{m-i})= r^*_B(\alpha_{m-i})-l^*_B(\alpha_{m-i})$ and $r^*_A(\alpha_{m-i})+l^*_A(\alpha_{m-i})= r^*_B(\alpha_{m-i})+l^*_B(\alpha_{m-i})$  for each $i< m$. Hence, ${\rcorte{A}{m}}-{\lcorte{A}{m}}< {\rcorte{B}{m}}-{\lcorte{B}{m}}$, ${\rcorte{A}{i}}+{\lcorte{A}{i}}= {\rcorte{B}{i}}+{\lcorte{B}{i}}$ and ${\rcorte{A}{i}}-{\lcorte{A}{i}}= {\rcorte{B}{i}}-{\lcorte{B}{i}}$ for all $i<m$. Therefore, in both cases, $\corteai{A}{m}\prec_{XY} \corteai{B}{m}$ and $\corteai{A}{i}= \corteai{B}{i}$ for each $i< m$. Since, clearly $m=m(A,B)$ it follows that $A\lhd_{XY}^S B$.
   
Reciprocally, if $A \lhd_{XY}^S B$ then  $\corteai{A}{m}\prec_{XY} \corteai{B}{m}$ where $m=m(A,B)$. So, either ${\lcorte{A}{m}}+{\rcorte{A}{m}}<{\lcorte{B}{m}}+{\rcorte{B}{m}}$ or ${\lcorte{A}{m}}+{\rcorte{A}{m}}={\lcorte{B}{m}}+{\rcorte{B}{m}}$ and ${\rcorte{A}{m}}-{\lcorte{A}{m}}<{\rcorte{B}{m}}+{\lcorte{B}{m}}$. In the first case, take $n_0=2m-1$ and in the second case $n_0=2m$. In any of the cases we have that $c_{n_0}(A)<c_{n_0}(B)$ and since for each $i<m$ we have that $\corteai{A}{i}= \corteai{B}{i}$ then $c_{i}(A)= c_{i}(B)$. Thereby, $A\leq_{WW}^S B$.
  \end{proof}
  
\begin{corolario}
For all upper dense sequence  $S=(\alpha_i)_{i\in\con{Z}^+}$ in $]0,1]$, the relation $\leq_{WW}^S$ is an admissible order on $\NFR$.
\end{corolario}
\begin{proof}
 Straightforward from Theorem \ref{ZBM} and Proposition \ref{pro-leq-WW-leq-XY}.
\end{proof}

 \begin{definicion}\label{positivonegativo}
  Let $\preceq$ be an admissible order on $\NFR$.
   A fuzzy number $A$ is $\preceq$-positive if $\crispy{0}\prec A$, is $\preceq$-negative if $A \prec \crispy{0}$, is non $\preceq$-negative if~ $\crispy{0}\preceq A$ and is non $\preceq$-positive if $A\preceq \crispy{0}$.
  \end{definicion}

  \begin{observacion}
Clearly, each fuzzy number $A$ is either $\preceq$-positive, $\preceq$-negative or $A=\crispy{0}$. Nevetheless, some fuzzy number are $\preceq_1$-positives for an admissible order $\preceq_1$ but are $\preceq_2$-negative for an admissible order $\preceq_2$.   For example, the fuzzy triangle number $(-1,0,2)$ is $\unlhd_{Lex1}^S$-negative and $\unlhd_{Lex2}^S$-positive for any  upper dense sequence $S=(\alpha_i)_{i\in\con{Z}^+}$  in $]0,1]$.
  
  \end{observacion}

\begin{observacion}
From Definition \ref{positivonegativo} and \cite[Pag. 104, point 5]{Klir1995}, if $B$ and $C$ are $\preceq$-positive then $A(B+C)=AB+AC$ for all $A\in\NFR$.
\end{observacion}

\begin{proposicion}\label{pro-preceq-positive-[0]}
 Let $A\in \NFR$. Then  $A$ is $\preceq$-positive for each admissible order $\preceq$ in $\NFR$ if and only if $A\neq \crispy{0}$ and $0\leq supp(A)^- $.
\end{proposicion}
\begin{proof}

Firstly, we assume $A=\crispy{0}$ then, trivially, $A$ is non $\preceq$-positive for each admissible order $\preceq$. If  $0> supp(A)^-$ then $[supp(A)^-,supp(A)^+]\prec_{Lex1} [0]$ and therefore, for any upper dense sequence $S=(\alpha_i)_{i\in\con{Z}}$  in $]0,1]$,  $\corteai{A}{m}\prec_{Lex1} [0]=\corteai{\crispy{0}}{m}$ for $m=m(\crispy{0},A)=\min\{i\in \con{Z}^+: \corteai{\crispy{0}}{i}\neq \corteai{A}{i}\}$. So, $A\lhd_{Lex1}^S\crispy{0}$ and therefore $A$ is not $\unlhd_{Lex1}^S$-positive. Thus,  this side of the proposition holds by contraposition.

On the other hand, if $A\neq \crispy{0}$ and $0\leq supp(A)^- $, then $\cortea{\crispy{0}} =[0]\leq_{KM} \cortea{A}$ and  $[0] <_{KM} \corteai{A}{m}$ for $m=m(\crispy{0},A)$. So, $\crispy{0}<_{KY} A$ and therefore for any admissible order $\preceq$ we have that $\crispy{0}\prec A$, that is  $A$ is $\preceq$-positive.
\end{proof}

\begin{proposicion}\label{pro-preceq-negative-[0]}
 Let $A\in \NFR$. Then  $A$ is $\preceq$-negative for each admissible order $\preceq$ in $\NFR$ if and only if $A\neq \crispy{0}$ and $supp(A)^+ < 0$.
\end{proposicion}
\begin{proof}
Using similar steps to Proposition \ref{pro-preceq-positive-[0]}, we obtain the result.
\end{proof}

\begin{corolario}
Let $A\in \NFR$ be such that $supp(A)^- <0<supp(A)^+$. Then there exist admissible orders $\preceq_1$ and $\preceq_2$ such that $A$ is $\preceq_1$-positive and $A$ is $\preceq_2$-negative
\end{corolario}

\begin{teorema}\label{teo-ao-positive}
 Let $A\in \NFR$ and $S=(\alpha_i)_{i\in\con{Z}^+}$ be an upper dense sequence in $]0,1]$. Then
 \begin{enumerate}
\item $A$ is $\unlhd_{Lex1}^S$-positive if and only if $0\leq \lcorte{A}{m}$ and $A \neq{\crispy{0}}$,
 
  \item $A$ is $\unlhd_{Lex2}^S$-positive  if and only if $ 0<{\rcorte{A}{m}}$,

  \item $A$ is $\unlhd_{XY}^S$-positive if and only if  $-{\lcorte{A}{m}} \leq \rcorte{A}{m}$ and $A\neq {\crispy{0}}$,
  
  \item $A$ is $\unlhd_{2XY}^S$-positive if and only if $A\neq {\crispy{0}}$  and  $-\lcorte{A}{m}\leq  3\rcorte{A}{m}$,
 \end{enumerate}
with $m=m(\crispy{0},A)$.
\end{teorema}

\begin{proof} Let $A\in \NFR$ and $m=m(\crispy{0},A)=\min\{i\in \con{Z}^+: \corteai{\crispy{0}}{i}\neq \corteai{A}{i}\}=\min\{i\in \con{Z}^+: [0]\neq \corteai{A}{i}\}$, we have:
 \begin{enumerate}
  \item[1)]  $A$ is $\unlhd_{Lex1}^S$-positive  if and only if  $\crispy{0}\lhd_{Lex1}^S A$ if and only if   $\corteai{\crispy{0}}{m}=[0]\prec_{Lex1} \corteai{A}{m}$  if and only if $0< {\lcorte{A}{m}}$ or, ${\lcorte{A}{m}}=0$ and $0< {\rcorte{A}{m}}$  if and only if $0\leq {\lcorte{A}{m}}$ and $A \neq \crispy{0}$.   
   
  \item[3)] $A$ is $\unlhd_{XY}^S$-positive  if and only if  $\corteai{\crispy{0}}{m}=[0]\prec_{XY} \corteai{A}{m}$ if and only if    
  $0< {\lcorte{A}{m}}+{\rcorte{A}{m}}$ or,  ${\lcorte{A}{m}}+{\rcorte{A}{m}}=0$ and  $0< {\rcorte{A}{m}}-{\lcorte{A}{m}}$   
if and only if $-{\lcorte{A}{m}} < {\rcorte{A}{m}}$ or,  $-{\lcorte{A}{m}}={\rcorte{A}{m}}$ and  ${\lcorte{A}{m}} <{\rcorte{A}{m}}$ 
if and only if $-{\lcorte{A}{m}} \leq {\rcorte{A}{m}}$ and $A\neq \crispy{0}$.
 \end{enumerate}
 The proof for {2 and 4} is analagous.
 \end{proof}

\begin{corolario}  Let $A\in \NFR$ and $S=(\alpha_i)_{i\in\con{Z}^+}$ be an upper dense sequence in $]0,1]$. If $A\neq\crispy{0}$ then  \begin{enumerate}   \item $A$ is $\unlhd_{Lex1}^S$-positive  if and only if $0\leq {\lcorte{A}{m}}$,   \item $A$ is $\unlhd_{Lex2}^S$-positive  if and only if $ 0<{\rcorte{A}{m}}$,   \item $A$ is $\unlhd_{XY}^S$-positive if and only if  $-{\lcorte{A}{m}} \leq {\rcorte{A}{m}}$,   \item $A$ is $\unlhd_{2XY}^S$-positive if and only if $-\lcorte{A}{m}\leq 3{\rcorte{A}{m}}$,  \end{enumerate} with $m=m(\crispy{0},A)$. \end{corolario}

A natural property of the arithmetic of real numbers is that the sum of two positive numbers is always positive and the sum of two negative numbers result is also a negative number. However, this natural property does not work for each admissible order considering the addition of fuzzy numbers given in Preliminary section.
Indeed, take as $S=(\alpha_i)_{i\in\con{Z}^+}$ be an upper dense sequence in $]0,1]$ such that $\alpha_1=0.8$ (By item 3 of Remark \ref{obs-S} such $S$ there exists), and consider the triangular fuzzy numbers $A=(-3,1,2)$ and  $B=(-4,0,2)$. Then,
${\corteai{A}{1}}=[0.2,1.2]$, $\corteai{B}{1}=[-0.8,0.4]$ and therefore both are $\unlhd_{2XY}^S$-positive. However, $\corteai{A+B}{1}=[-0.6,1.6]$ and hence $A+B$ is $\unlhd_{2XY}^S$-negative.

So, the next result analyses which of the other three admissible orders verifies this natural property.

\begin{proposicion}
 Let $S=(\alpha_i)_{i\in\con{Z}^+}$ be an upper dense sequence in $]0,1]$ and $A,B\in \NFR$. Then
 \begin{enumerate}
 \item If $A$ and $B$ are $\unlhd_{Lex1}^S$-positive then $A+B$ is also $\unlhd_{Lex1}^S$-positive;
 \item If $A$ and $B$ are $\unlhd_{Lex1}^S$-negative then $A+B$ is also $\unlhd_{Lex1}^S$-negative;
 \item If $A$ and $B$ are $\unlhd_{Lex2}^S$-positive then $A+B$ is also $\unlhd_{Lex2}^S$-positive;
 \item If $A$ and $B$ are $\unlhd_{Lex2}^S$-negative then $A+B$ is also $\unlhd_{Lex2}^S$-negative;
 \item If $A$ and $B$ are $\unlhd_{XY}^S$-positive then $A+B$ is also $\unlhd_{XY}^S$-positive;
  \item If $A$ and $B$ are $\unlhd_{XY}^S$-negative then $A+B$ is also $\unlhd_{XY}^S$-negative.
  \end{enumerate}
\end{proposicion}
\begin{proof} Let $A,B\in \NFR$ and $m_A=m(\crispy{0},A)$, $m_B=m(\crispy{0},B)$. From Theorem \ref{teo-ao-positive} {we have:}
\begin{enumerate}
\item[1.] If  $A$ and $B$ are $\unlhd_{Lex1}^S$-positive then {by Theorem \ref{teo-ao-positive}-(1) we have that $A\neq \crispy{0}$, $B\neq \crispy{0}$, $0\leq \lcorte{A}{m_A}$ and $0\leq \lcorte{B}{m_B}$}. 

Without loss of generality we can suppose that $m_A\leq m_B$. Then, for each {$i< m_A$}, by Eq. \eqref{eq-alpha-cuts-A*B} we have that $\corteai{A+B}{i}=\corteai{A}{i}+\corteai{B}{i}=[0]+[0]=[0]$ and, since {$0\leq \lcorte{B}{m_B}$}, then {$0\leq \lcorte{A}{m_A} \leq \lcorte{A}{m_A}+\lcorte{B}{m_A}
  =\lcorte{A+B}{m_A}$.} Therefore, {$m_A=m(A+B,\crispy{0})$} and  thereby $\crispy{0}\lhd_{Lex1}^S A+B$.

\item[{3.}] If  $A$ and $B$ are $\unlhd_{Lex2}^S$-positive then {$0 < \rcorte{A}{m_A}$} and {$0 < \rcorte{B}{m_B}$.}
Without loss of generality we can suppose that $m_A\leq m_B$. Then, for each $i<m_A$, by Eq. \eqref{eq-alpha-cuts-A*B} we have that $\corteai{A+B}{i}=\corteai{A}{i}+\corteai{B}{i}=[0]+[0]=[0]$ and since {$\rcorte{A+B}{m_A}=\rcorte{A}{m_A}+\rcorte{B}{m_A} > 0$} then $m_A=m(A+B,\crispy{0})$ and therefore, $\crispy{0}\lhd_{Lex2}^S A+B$.

\item[{5.}] If  $A$ and $B$ are $\unlhd_{XY}^S$-positive then ${-\lcorte{A}{m_A} < \rcorte{A}{m_A}}$ and $-\lcorte{B}{m_B} < \rcorte{B}{m_B}$. 
Without loss of generality we can suppose that $m_A\leq m_B$. Then, for each $i<m_A$, by Eq. \eqref{eq-alpha-cuts-A*B} we have that $\corteai{A+B}{i}=\corteai{A}{i}+\corteai{B}{i}=[0]+[0]=[0]$ and since $\lcorte{A+B}{m_A}=\lcorte{A}{m_A}+\lcorte{B}{m_A}$  and $\rcorte{A+B}{m_A}=\rcorte{A}{m_A}+\rcorte{B}{m_A}$ then $\lcorte{A+B}{m_A}< \rcorte{A+B}{m_A}$. So,
$m_A=m(A+B,\crispy{0})$ and therefore, $\crispy{0}\lhd_{XY}^S A+B$.
\end{enumerate}

The proof for {2, 4 and 6} is analagous.
\end{proof}

\section{Ranking Path Costs in Fuzzy Weighted Graphs}

Weighted graphs arise from the necessity to model practical problems where the edges in a graph have an associated  cost as, for example, the well-known problem of the {traveling} salesman. This problem consists of going through a list of towns, visiting each one exactly once and featurning to the origin and in such a way that the travelling total time (the cost) is minimized \cite{Bondy2008}. When we consider that such cost is imprecise, as the time dispensed in the travel of car from a city $X$ to a city $Y$, the use of $\NFR$ to model the costs is more appropriate. 

\begin{definicion} \cite{Cornelis2004}
A fuzzy weighted graph is a triple  $G=\prt{V,E,c}$ where $V$ is a set whose elements are called vertices, $E\subseteq V\times V$ is a set of edges and $c:E\rightarrow \NFR$ is the cost (or weight) function. Given $v,u\in V$, a $(v,u)$-path  in $G$ is a finite and non empty sequence of edges  $p=(e_1,\ldots,e_n)=((v_1,u_1),\ldots,(v_n,u_n))$ such that $v_i=u_{i-1}$ for each $i=2,\ldots,n$, $v_1=v$ and $u_n=u$. A $(v,u)$-path is a cycle if $v=u$.
\end{definicion}

In order to simplify the notation, we will denote the $(v_1,u_n)$-path $p=((v_1,u_1),(v_2,u_2),\ldots,(v_n,u_n))$ by $p=(v_1,\ldots,v_n,u_n)$.

A fuzzy weighted graph  $G=\prt{V,E,c}$ such that $E$ is symmetric, i.e. $(v,u)\in E$ if and only if $(u,v)\in E$, and $c(v,u)=c(u,v)$ for each $(v,u)\in E$ will be called of undirected fuzzy weighted graph.

\begin{ejemplo}
\label{ejem-FWG}
The Figure \ref{graph1} is an example of a directed and of an undirected fuzzy weighted graph with triangular fuzzy numbers as cost.
\end{ejemplo}
\begin{figure}
\footnotesize
\centering
  \subfloat[]{
   \label{diagramaa}
    \includegraphics[width=0.3\textwidth]{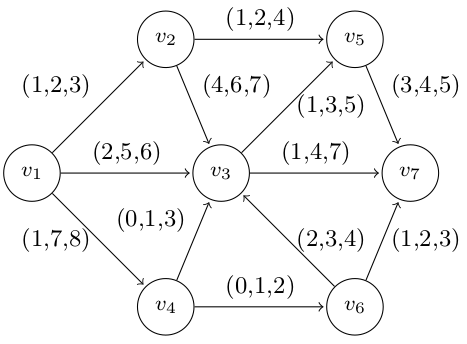}}\\
      \subfloat[]{
   \label{diagramab}
 \hspace{0.8cm}   \includegraphics[width=0.3\textwidth]{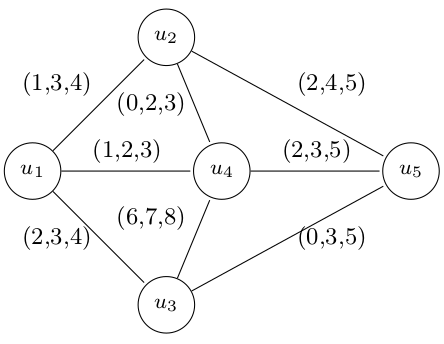}}
           \label{f:numeros}
\caption{Examples of directed and undirected fuzzy weighted graphs.}
\label{graph1}
\end{figure}
The cost of a $(v,u)$-path $p=(e_1,\ldots,e_n)$, denoted by $c(p)$ is given by the addition of the cost of each edge in the path, i.e. $c(p)=\sum\limits_{i=1}^n c(e_i)$. Given a pair of vertices $(v,u)$ in  a fuzzy weighted graph $G$ there may exist several or no $(v,u)$-path. Given an admissible order $\preceq$ on $\NFR$, a fuzzy weighted graph  $G=\prt{V,E,c}$, and $v,u\in V$, we say that a $(v,u)$-path $p$ is $\preceq$-minimal if $c(p)\preceq c(q)$ for each $(v,u)$-path $q$ in 
$G$.

\begin{ejemplo}
In the case of the (directed) fuzzy weighted graph in Figure \ref{graph1}, the Table \ref{tab-ej} presents all the possible  paths from $v_1$ to $v_7$ and their costs.
Given an  upper dense sequence  $S=(\alpha_i)_{i\in\con{Z}^+}$ in $]0,1]$ such that $\alpha_1=0.8$ we have the following ranking with respect to the orders:

\ITEM For $\unlhd_{Lex1}^S$:
\begin{multline*}
p_1\lhd_{Lex1}^S p_4  \lhd_{Lex1}^S p_6 \unlhd_{Lex1}^S p_7 \lhd_{Lex1}^S p_5\\
\lhd_{Lex1}^S p_3 \lhd_{Lex1}^S p_2 \lhd_{Lex1}^S p_9 \lhd_{Lex1}^S p_8 \lhd_{Lex1}^S p_{10}.
\end{multline*}
\ITEM For $\unlhd_{Lex2}^S$:
\begin{multline*}
p_1\lhd_{Lex2}^S p_4  \lhd_{Lex2}^S p_6 \lhd_{Lex2}^S p_5 \lhd_{Lex2}^S p_7\\
 \lhd_{Lex2}^S p_3 \lhd_{Lex2}^S p_2 \lhd_{Lex2}^S p_9 \lhd_{Lex2}^S p_8 \lhd_{Lex2}^S p_{10}.
\end{multline*}
\ITEM For $\unlhd_{XY}^S$:
\begin{multline*}
p_1\lhd_{XY}^S p_4  \lhd_{XY}^S p_6 \lhd_{XY}^S p_7 \lhd_{XY}^S p_5\\
\lhd_{XY}^S p_3 \lhd_{XY}^S p_2 \lhd_{XY}^S p_9 \lhd_{XY}^S p_8 \lhd_{XY}^S p_{10}.
\end{multline*}

Observe that ranking based on $\unlhd_{Lex1}^S$ and $\unlhd_{XY}^S$ is the same and differs from the given by $\unlhd_{Lex2}^S$ in the positions of $p_5$ and $p_7$. So,  the minimal path from $V_1$ to $V_7$, with respect to all the three admissible orders,  is $p_1$.

\begin{table}[h!]
\begin{center}
\caption{Paths of the fuzzy weighted graph in Figure \ref{diagramaa}.}
\label{tab-ej}
\begin{tabular}{c|c|c|c}
\textbf{Identification} & \textbf{Paths} & \textbf{Cost} & \textbf{$\alpha_1$-cut} \\
\hline
$p_1$ & $(v_1,v_2,v_5,v_7)$ & $(5,8,12)$ & $[7.4,8.8]$ \\
$p_2$ & $(v_1,v_2,v_3,v_5,v_7)$ & $(9,15,20)$ & $[12.2,16]$ \\
$p_3$ & $(v_1,v_2,v_3,v_7)$ & $(6,12,17)$ & $[10.8,16]$ \\
$p_4$ & $(v_1,v_3,v_7)$ &  $(3,9,13)$ & $[7.8,9.8]$ \\
      $p_5$ & $(v_1,v_3,v_5,v_7)$ & $(6,12,16)$ & $[10.8,12.8]$ \\
      $p_6$ & $(v_1,v_4,v_6,v_7)$ & $(2,10,13)$ & $[8.4,10.6]$ \\
      $p_7$ & $(v_1,v_4,v_3,v_7)$ & $(2,12,18)$ & $[10,13.2]$ \\
      $p_8$ & $(v_1,v_4,v_3,v_5,v_7)$ & $(5,15,21)$ & $[13,16.2]$ \\
      $p_9$ & $(v_1,v_4,v_6,v_3,v_7)$ & $(4,15,21)$ & $[12.8,16.2]$\\
      $p_{10}$ & $(v_1,v_4,v_6,v_3,v_5,v_7)$ & $(7,18,24)$ & $[15.8,19.2]$ \\ \hline
    \end{tabular}
  \end{center}
\end{table}
We note, according to a Figure \ref{diagramaa}, that we obtain
$$W=\tiny{\left(\begin{array}{ccccccc}
0 &(1,2,3)&(2,5,6)&(1,7,8)&\infty &\infty &\infty\\
\infty & 0 & (2,6,7)&\infty & (1,2,4) & \infty  & \infty\\ \infty & \infty & 0 & \infty & (1,3,5) & \infty & (1,4,7)\\ \infty & \infty & (0,1,3)& 0 &\infty & (0,1,2)&\infty\\ \infty & \infty &\infty & \infty & 0 & \infty & (3,4,7)\\ \infty & \infty & (2,3,4) & \infty & \infty & 0 & (1,2,3)\\ \infty & \infty &\infty & \infty &\infty & \infty &\infty
\end{array}\right)}$$
where $\infty$ denotes that there is no path and $0$ is the null distance. 
\end{ejemplo}

In the following we present an algorithm to determine the $\preceq$-minimal $(v,u)$-path in a fuzzy weighted graph  $G=\prt{V,E,c}$ which is based on an adaptation of Floyd-Warshall algorithm (see \cite{RFloyd,cormen}) 

\begin{algorithm}[H]\small
\algorithmicrequire{ A fuzzy wegthed graph $G=\prt{V,E,c}$ and $v_s,v_f\in V$ with $V=\{v_1,....,v_n\}$}\\
\algorithmicensure{ Solution alternative: vector $p=min-path(v_s,v_f)$ and $d_{s,f}^{(n)}$}\\
\caption{Minimal path between two nodes in a fuzzy weighted graph}
\label{Algorithm M}
\begin{algorithmic}[1]
\STATE $M = \widetilde{1}+\sum_{(v_i,v_j)\in E} c(v_i,v_j) $
\STATE for $i=1$ to $n$
\STATE ~~~~~for $j=1$ to $n$
\STATE ~~~~~~~  if $i=j$ 
\STATE ~~~~~~~  then $d^{(0)}_{ij}= \widetilde{0}$ 
\STATE ~~~~~~~~~~ $ \pi^{(0)}_{ij}=NIL$
\STATE ~~~~~~~  else if $(v_i,v_j)\in E$
\STATE ~~~~~~~~~~ then $d^{(0)}_{ij}=c(v_i,v_j)$
\STATE ~~~~~~~~~~~~~   $\pi^{(0)}_{ij}=i$
\STATE ~~~~~~~~~~  else $d^{(0)}_{ij}=M$
\STATE ~~~~~~~~~~~~~   $\pi^{(0)}_{ij}=NIL$
\STATE for $k=1$ to $n$
\STATE ~~~~~ for $i=1$ to $n$
\STATE ~~~~~~~~~ for $j=1$ to $n$
\STATE ~~~~~~~~~~~~~ If $d^{(k-1)}_{ij}\preceq d^{(k-1)}_{ik}+d^{(k-1)}_{kj}$
\STATE ~~~~~~~~~~~~~~~~ then $d^{(k)}_{ij}= d^{(k-1)}_{ij}$
\STATE ~~~~~$\pi^{(k)}_{ij}=\pi^{(k-1)}_{ij}$
\STATE ~~~~~~~~~~ else $d^{(k)}_{ij}= d^{(k-1)}_{ik}+d^{(k-1)}_{kj}$
\STATE ~~~~~~~~~~~~~ $\pi^{(k)}_{ij}=\pi^{(k-1)}_{kj}$
\STATE    $k=1$
\STATE    $m=2$
\STATE    Repeat
\STATE    ~~~ While $\pi_{p(k) p(k+1)}^{(n)}\neq p(k)$
\STATE    ~~~~~~~ for $i=k+1$ to $m$ 
\STATE    ~~~~~~~~~~~ $p(m-i+k+2)=p(m-i+k+1)$
\STATE    ~~~~~~~ $p(k+1)=\pi_{p(k) p(k+1)}^{(n)}$
\STATE    ~~~~~~~ $m=m+1$
\STATE    ~~~ $k=k+1$
\STATE    Until $k=m$
\STATE Return $p$, $d^{(n)}_{sf}$

\end{algorithmic}
\end{algorithm}

\section{Illustrative Example}
Consider the road distances between the neighbor capitals of the 9 Brazilian Northeast States obtained from the sites 
\begin{enumerate}
\item \url{http://www.distanciasentrecidades.com/},
\item \url{https://www.rotamapas.com.br/},
\item \url{https:// www.melhoresrotas.com/s/distancia-entre-cidades},
\item \url{http://rotasbrasil.com.br},
\end{enumerate}
From such site we observed, for example, that the distance between the cities of Natal and Jo\~ao Pessoa vary (178 km, 179 km, 181 km, 182 km, 189 km, 214 km). From this date we generate the triangular fuzzy number $(178, 181.5, 214)$ by taking the minimum, median and maximum of such distances, i.e., $(\min,$ $\Me,$ $\max)$. {In the following Tables and Figure \ref{graph_capitales_nordeste} the abbreviations of the cities will be used, i.e., S\~ao Luis by SL, Teresina by T, Fortaleza by F and so on. There are 362,880 possible routes for the travelling salesman problem.}

   \begin{table*}[h!]
\begin{center}
\begin{tiny}
\begin{tabular}{c|c|c|c|c|c|c|c|c|c}
Km & SL & T & F& N & JP & R & M & A & S\\
\hline 
 &  & $(433,$ & $(878,$ & $(1395,$ & $(1549,$ & $(1557,$& $(1552,$ & $(1563,$ & $(1627,$\\
SL &  & $439,$ &$1169.1,$ & $1408,$ & $1556,$ & $1664,$& $1563,$ & $1584,$ & $1658,$\\
  &  & $610.8)$ & $1177.9)$ & $1753.9)$ & $1825.6)$ & $1738.4)$& $1733.3)$ & $1819)$ & $1834.7)$\\
\hline 
  & $(432,$ &  & $(592,$ & $(1040,$ & $(1156,$ &  $(1127,$ & $(1125.7,$ & $(1133,$ & $(1196,$\\
T  & $435,$ &  & $593.5,$ & $1097.6,$ & $1160.8,$ &  $1127.5,$  & $1126.5,$ & $1143.15,$ & $1226.55,$\\
  & $609.1)$ &  & $594.4)$ & $1146.2)$ & $1251.5)$ & $1130.7)$ & $1139)$ & $1159)$ & $1233)$\\
\hline 
 & $(879,$ & $(592,$ && $(511,$ & $(665,$ & $(771.9,$ &  $(957.4,$ & $(1089.9,$ & $(1190,$\\
F & $900,$ & $593.9,$ && $515,$ & $668,$ & $773,$
&  $1027,$ & $1119,$ & $1203.25,$\\
 & $1179)$ & $594)$ && $531.8)$ & $698.3)$ & $798.5)$
 &  $1050.5)$ & $1130)$ & $1205)$\\
\hline 
 & $(1396,$ & $(1041,$ & $(511,$ &  & $(178.8,$ & $(285,$
  & $(535,$ & $(782,$&  $(1093,$ \\
N & $1418,$ & $1107,$ & $ 520,$&  & $180,$ & $285.1,$
 & $539.85,$ & $788.5,$& $1095,$  \\
 & $1756.9)$ & $1149)$ & $530)$ &  & $182)$ & $286)$
  & $543)$ & $794.8)$&  $1101.5)$  \\
\hline 
& $(1550,$ & $(1155,$ & $(665,$ & $(175.9,$ &  & $(116,$ & $(366,$ & $(613,$ &  $(924,$ \\
 JP & $1571,$ & $1158.8,$ &  $671.05,$ & $180,$ &  & $116.6,$ & $370.85,$ & $620.5,$ &  $927,$ \\
 & $1831.4)$ & $1257)$ &  $684)$& $182)$ &  & $117)$ & $375)$ & $625.8)$ &  $932.4)$ \\
\hline 
& $(1557,$ & $(1126,$ & $(773,$& $(285.8,$ & $(116,$ & & $(154.7,$& $(497,$ & $(808,$ \\
R& $1673.5,$ & $1127.5,$ & $781,$& $286,$ & $118,$ &
& $253,$& $504.9,$ & $811,$  \\
& $1735.6)$ & $1127.7)$ & $803.2)$& $290)$ & $119.2)$ & 
& $258)$& $510)$ & $816.4)$ \\
\hline 
& $(1552,$ & $(1124,$ & $(964,$ & $(537,$ & $(367,$ & $(251,$ &  & $(271,$ & $(581,$ \\
M& $1561.5,$ & $1126.2,$ & $1028,$ & $539.05,$ & $371.25,$ & $253.8,$ &  & $279.25,$ & $588.1,$\\
& $1733.2)$ & $1137)$ & $1051.2)$ & $543)$ & $373)$ & $258)$
&  & $291)$ & $593)$ \\
\hline 
 & $(1561,$ & $(1131,$ & $(1095,$ & $(782,$ & $(613,$ & $(497,$ & $(270,$ &  & $(325,$\\
A & $1583,$ & $1142.05,$ & $1119.5,$ & $783.4,$ & $615.1,$ & $497.65,$ & $271,$ &  & $325.95,$\\
 & $1821.7)$ & $1158)$ & $1130)$ & $794)$ & $624)$ & $508)$  & $289)$ &  &  $338)$\\
\hline 
& $(1626,$ & $(1196,$ & $(1188,$ & $(1091,$ & $(922,$ & $(805.6,$ & $(580,$ & $(321.8,$ &  \\
S& $1655,$ & $1222.75,$ & $1200.4,$& $1093.55,$ & $924.75,$ & $807.5,$  & $581,$ & $326,$ & \\
& $1831.3)$ & $1232)$ & $1204)$& $1096)$ & $926)$ & $810)$  & $591)$ & $349)$ &  \\
\hline
Km & SL & T & F& N & JP & R & M & A & S
\end{tabular}\end{tiny}
\caption{Triangular fuzzy numbers defined by $(\min,\Me,\max)$ for distances between capitals in the Brazilian Northeast 
.}
\label{tab-numeros-capitales1}
  \end{center}
   \end{table*} 

\begin{figure}[h]
\centering
\begin{tikzpicture}
\node[draw,circle] (SL) at (0,7) {SL};
\node[draw,circle] (T) at (1.5,5) {T};
\node[draw,circle] (F) at (4.37777777333,6.11111111107) {F};
\node[draw,circle] (N) at (6.82222222222,5
) {N};
\node[draw,circle] (JP) at (7.066666667,4.024444444) {JP};
\node[draw,circle] (R) at (7,3
) {R};
\node[draw,circle] (M) at (6.222222222222222,2.222222222222) {M};
\node[draw,circle] (A) at (5.17777777777778,1.4666666666666667) {A};
\node[draw,circle] (S) at (4.155555555555,0) {S};
\draw[<->] (SL) to[bend left=0]
(T);
\draw[<->] (SL) to[bend left] 
(F);
(S);
\draw[<->] (T) to[bend left]
 (F);
\draw[<->] (T) to[bend left=15]
(JP);
\draw[<->] (T) to[bend left=-5]
(R);
\draw[<->] (T) to[bend left=-5]
(M);
\draw[<->] (T) to[bend right=10]
 (A);
\draw[<->] (T) to[bend right=10]
(S);
\draw[<->] (F) to[bend left]
(N);
\draw[<->] (N) to[bend left]
(JP);
\draw[<->] (JP) to[bend left]
(R);
\draw[<->] (R) to[bend left]
(M);
\draw[<->] (M) to[bend left]
(A);
\draw[<->] (A) to[bend left]
(S);
\end{tikzpicture}
\caption{Examples of directed weighted graphs.}
\label{graph_capitales_nordeste}
\end{figure}

From the Table \ref{tab-numeros-capitales1} we observe that the distance from Fortaleza to Jo\~ao Pessoa is different from Fortaleza to Natal and from Natal to Jo\~ao Pessoa, we have: 
\begin{align*}
(665,668,698.3)&\neq (511,515,531.8)+(178.8,180,182)\\
&=(689.8,695,713.5),
\end{align*}
from where $(665,668,698.3)<_{KY}(689.8,695,713.5)$. Then by definition of admissible order $(665,668,698.3)\prec(689.8,695,713.5)$.

In the case of the (directed) fuzzy weighted graph in Figure \ref{graph_capitales_nordeste},  the Table \ref{tab-cidades} presents some routes chosen possibilities among the 362,880 randomly and their cost from any capital city of the Brazilian northeast until returning to it, having passed through all the other capitals.
\begin{table*}[h!]
\begin{center}
\begin{tabular}{c|c|c|c}
Route & Paths & Cost & $\alpha$-cut\\
\hline 
$r_1$ &\tamano{6}{
SL,JP,R,S,M,F,A,T,N,SL} & \tamano{6}{$(8673.9,8869.25,9592.3)$} & \tamano{6}{$[195.35\alpha+8673.9,9592.3-723.05\alpha]$}\\
\hline 
$r_2$ & \tamano{6}{N,T,SL,F,A,R,JP,S,M,N} & \tamano{6}{$(6094.9,6492.84,6759.6)$} & \tamano{6}{$[397.94\alpha +6094.9,6759.6-266.76\alpha]$}\\
\hline 
$r_3$ & \tamano{6}{N,F,JP,T,A,R,S,SL,M,N} & \tamano{6}{$(8484,8653.85,8983.3)$} & \tamano{6}{$[169.85\alpha+8484,8983.3-329.45\alpha]$}\\
\hline
$r_4$ & \tamano{6}{N,JP,A,F,SL,S,M,T,R,N} & \tamano{6}{$(7509.6,7598.7,8100)$} &
\tamano{6}{$[89.1\alpha +7509.6,8100-501.3\alpha]$}\\
\hline
$r_5$& \tamano{6}{JP,SL,A,F,T,R,N,M,S,JP} &\tamano{6}{$(8250.8,8334.6,8857.1)$}  &
\tamano{6}{$[83.8\alpha+8250.8,8857.1-522.5\alpha]$}\\
\hline
$r_6$ & \tamano{6}{JP,T,F,R,A,M,S,SL,N,JP} &\tamano{6}{$(7066.7,7132,7809.1)$} &
\tamano{6}{$[65.3\alpha+7066.7,7809.1-677.1\alpha]$}\\
\hline
$r_7$ & \tamano{6}{JP,N,S,T,M,F,A,R,SL,JP} & \tamano{6}{$(9247.5,9498.4,9904.9)$} &
\tamano{6}{$[250.9\alpha+9247.5,9904.9-406.5\alpha]$} \\
\hline
$r_8$ & \tamano{6}{JP,S,N,A,R,F,SL,T,M,JP} &\tamano{6}{$(6871.7,6924.45,7436.2)$} &
\tamano{6}{$[52.75\alpha +6871.7 ,7436.2-511.75\alpha)]$} \\
\hline
$r_9$ & \tamano{6}{A,JP,T,N,M,SL,S,R,F,A} & \tamano{6}{$(9190.5,9333.35,9881.3)$} & 
\tamano{6}{$[142.85\alpha +9190.5,9881.3-547.95\alpha]$}\\
\hline
$r_{10}$ & \tamano{6}{S,T,M,A,F,JP,SL,R,N,S} & \tamano{6}{$(8838.5,9092,9351.6)$} &
\tamano{6}{$[253.5\alpha+8838.5,9351.6-259.6\alpha]$}\\
 \hline
$r_{11}$ &\tamano{6}{N,R,F,JP,M,SL,T,S,A,N} & \tamano{6}{$(6988.7,7051.05,7638.7)$} &
 \tamano{6}{$[62.35\alpha+6988.7,7638.7-587.65\alpha]$}\\
 \hline
$r_{12}$ &\tamano{6}{A,N,JP,S,SL,R,M,T,F,A} & \tamano{6}{$(8028.4,8301.1,8597.5)$} &
 \tamano{6}{$[272.7\alpha +8028.4,8597.5-296.4\alpha]$}\\
  \hline
$r_{13}$ &\tamano{6}{M,SL,F,S,R,JP,T,A,N,M} & \tamano{6}{$(8681.6,9073.55,9341.3)$} &
 \tamano{6}{$[391.95\alpha+8146.6,9341.3-267.75\alpha]$}\\
 \hline
$r_{14}$ &\tamano{6}{F,R,T,SL,N,JP,S,A,M,F} & \tamano{6}{$(6383.5,6475.5,7092.8)$} &
\tamano{6}{$[92\alpha +6383.5,7092.8-617.3\alpha]$}\\
 \hline
 $r_{15}$ &\tamano{6}{M,F,R,S,JP,T,SL,A,N,M} & \tamano{6}{$(7932.9,8037.8,8612.2)$} &
 \tamano{6}{$[104.9\alpha+7932.9,8612.2-574.4\alpha]$}\\
 \hline
 $r_{16}$ &\tamano{6}{M,A,N,S,JP,SL,F,T,M} & \tamano{6}{$(7213.7,7592.9,7854.8)$} &
 \tamano{6}{$[379.2\alpha +7213.7,7854.8-261.9\alpha]$}\\
 \hline
$r_{17}$ &\tamano{6}{SL,S,R,F,N,JP,A,M,T,SL} & \tamano{6}{$(6334.4,6394.2,6822.6)$} &
\tamano{6}{$[59.8\alpha +6334.4,6822.6-428.4\alpha]$}\\
\hline
$r_{18}$ &\tamano{6}{A,N,M,T,R,SL,F,JP,S,A} & \tamano{6}{$(7918.3,8340.05,8497.9)$} &
\tamano{6}{$[421.75\alpha +7918.3,8497.9-157.85 \alpha]$}\\
 \hline
  $r_{19}$ &\tamano{6}{F,R,N,SL,M,S,JP,T,A,F} & \tamano{6}{$(8028.4,8301.1,8597.5)$} &
  \tamano{6}{$[272.7\alpha + 8028.4,8597.5-296.4\alpha]$}\\
  \hline
$r_{20}$ &\tamano{6}{F,M,S,T,A,SL,JP,R,N,F} & \tamano{6}{$(7890.2,8040.6,8618.8)$} &
\tamano{6}{$[150.4\alpha+7890.2,8618.8-578.2\alpha]$}\\
\end{tabular}
\caption{Paths with their cost of the fuzzy weighted graph in Figure \ref{graph_capitales_nordeste} and $\alpha\in~ ]0,1]$.}
\label{tab-cidades}
  \end{center}
\end{table*}

We note that of the possible routes of the Table \ref{tab-cidades} $r_1$ and $r_{10}$ cannot be compared with the partial order $\leq_{KY}$ since, if $\alpha_0=\frac{938}{2027}$, then $r_1\parallel r_{10}$ because, $\cortea{r_1}\leq_{KM}\cortea{r_{10}}$ for $\alpha\in~\left]\alpha_0,1\right]$ and $\cortea{r_{10}}\subseteq \cortea{r_1}$  for $\alpha\in~\left]0,\alpha_0\right[$.

In Figure \ref{reticompleto} we have ordered the costs in Table \ref{tab-cidades} from the smallest to the greatest, according to the $\le_{KY}$ partial order.
\begin{figure}[h]
\centering
 \begin{tikzpicture}
\node[draw,circle] (r_1) at (1,9) {$r_1$};
\node[draw,circle] (r_7) at (0,11) {$r_7$};
\node[draw,circle] (r_9) at (0,10) {$r_9$};
\node[draw,circle] (r_{10}) at (-2,9.5) {$r_{10}$};
\node[draw,circle] (r_{13}) at (-1,9) {$r_{13}$};
\node[draw,circle] (r_{19}) at (2,9.5) {$r_{19}$};
\node[draw,circle] (r_3) at (0,8) {$r_3$};
\node[draw,circle] (r_5) at (-1.5,7) {$r_5$};
\node[draw,circle] (r_{12}) at (1.5,6.5) {$r_{12}$};
\node[draw,circle] (r_{15}) at (0.5,6.5) {$r_{15}$};
\node[draw,circle] (r_{18}) at (-0.5,6.5) {$r_{18}$};
\node[draw,circle] (r_{20}) at (-1.5,6) {$r_{20}$};
\node[draw,circle] (r_4) at (0,5) {$r_4$};
\node[draw,circle] (r_2) at (-1,3.5) {$r_2$};
\node[draw,circle] (r_8) at (0,4) {$r_8$};
\node[draw,circle] (r_{14}) at (1,3.5) {$r_{14}$};
\node[draw,circle] (r_{17}) at (1,2.5) {$r_{17}$};
\draw[-] (r_7) to (r_9);
\draw[-] (r_9) to (r_{19});
\draw[-] (r_9) to (r_{10});
\draw[-] (r_{10}) to (r_{13});
\draw[-] (r_{19}) to (r_1);
\draw[-] (r_{13}) to (r_3);
\draw[-] (r_1) to (r_3);
\draw[-] (r_3) to (r_5);
\draw[-] (r_3) to (r_{18});
\draw[-] (r_3) to (r_{15});
\draw[-] (r_3) to (r_{12});
\draw[-] (r_5) to (r_{20});
\draw[-] (r_{20}) to (r_4);
\draw[-] (r_{18}) to (r_4);
\draw[-] (r_{15}) to (r_4);
\draw[-] (r_{12}) to (r_4);
\draw[-] (r_4) to (r_8);
\draw[-] (r_8) to (r_2);
\draw[-] (r_8) to (r_{14});
\draw[-] (r_{14}) to (r_{17});

\end{tikzpicture}
\caption{
In Table \ref{tab-cidades}, we have $r_4\leq_{KY}r_{16}\leq_{KY}r_6\leq_{KY}r_{11}\leq_{KY}r_8$.}
   \label{reticompleto}
\end{figure}

We note that the routes $r_1$, $r_2$, $r_5$, $r_{10}$, $r_{12}$, $r_{13}$, $r_{14}$, $r_{15}$, $r_{17}$, $r_{18}$, $r_{19}$ and $r_{20}$ are incomparable with at least another route.

Given an  upper dense sequence  $S=(\alpha_i)_{i\in\con{Z}^+}$ in $]0,1]$ such that $\alpha_1=0.8$ we have the following ranking with respect the orders:

\ITEM For $\unlhd_{Lex1}^S$:
\begin{multline*}
r_{17}\lhd_{Lex1}^Sr_2\lhd_{Lex1}^S r_{14}\lhd_{Lex1}^S r_8
\lhd_{Lex1}^Sr_{11}
\unlhd_{Lex1}^Sr_6\\
\lhd_{Lex1}^Sr_{16} \lhd_{Lex1}^S r_4 \lhd_{Lex1}^S r_{15}\lhd_{Lex1}^Sr_{20}\unlhd_{Lex1}^S r_{12}\\
\lhd_{Lex1}^S r_{18}
\lhd_{Lex1}^Sr_{5}
\lhd_{Lex1}^S r_3 \lhd_{Lex1}^Sr_{19}
\lhd_{Lex1}^S r_{13}\\
\lhd_{Lex1}^S r_1\lhd_{Lex1}^Sr_{10} \lhd_{Lex1}^S r_9\lhd_{Lex1}^S r_7.
\end{multline*}
\ITEM For $\unlhd_{Lex2}^S$:
\begin{multline*}
r_{17}\unlhd_{Lex2}^Sr_2\lhd_{Lex2}^S r_{14}\lhd_{Lex2}^S r_8\lhd_{Lex2}^Sr_{11}
\lhd_{Lex2}^Sr_6\\
\lhd_{Lex2}^Sr_{16}\lhd_{Lex2}^Sr_4
\lhd_{Lex2}^S r_{15}
\lhd_{Lex2}^Sr_{20}\lhd_{Lex2}^S r_{12}\\
\lhd_{Lex2}^S r_{18} \lhd_{Lex2}^Sr_{5}\lhd_{Lex2}^S r_3\lhd_{Lex2}^Sr_{19} \lhd_{Lex2}^S r_{1}\\
\lhd_{Lex2}^S r_{13} \lhd_{Lex2}^Sr_{10}
\lhd_{Lex2}^S r_9\lhd_{Lex2}^S r_7.
\end{multline*}
\ITEM For $\unlhd_{XY}^S$:
\begin{multline*}
r_{17} \lhd_{XY}^S r_2 \unlhd_{XY}^S  r_{14} \unlhd_{XY}^S r_8\unlhd_{XY}^S r_{11} \unlhd_{XY}^Sr_6\\
 \unlhd_{XY}^S r_{16} \unlhd_{XY}^S r_4  \unlhd_{XY}^S  r_{20}
\unlhd_{XY}^S r_{15} \unlhd_{XY}^S r_{12}\\
\unlhd_{XY}^Sr_{18} \unlhd_{XY}^Sr_5  \unlhd_{XY}^S r_3 \unlhd_{XY}^S r_{19}\unlhd_{XY}^S  r_{13}\\
 \unlhd_{XY}^S r_1\unlhd_{XY}^S r_{10}
 \unlhd_{XY}^S r_9 \unlhd_{XY}^S r_7.
\end{multline*}
\ITEM For $\unlhd_{2XY}^S$:
\begin{multline*}
r_{17} \lhd_{2XY}^S r_2 \lhd_{2XY}^S  r_{14} \lhd_{2XY}^S r_8
 \lhd_{2XY}^Sr_{11}\lhd_{2XY}^S r_{6}\\
\lhd_{2XY}^S r_{16} \lhd_{2XY}^S r_4 
\lhd_{2XY}^S r_{15}
\lhd_{2XY}^S r_{20} \lhd_{2XY}^S r_{12}\\
\lhd_{2XY}^S r_{18} \lhd_{2XY}^Sr_8 \lhd_{2XY}^S r_3
\lhd_{2XY}^S r_{19}
\lhd_{2XY}^S r_{13}\\
 \lhd_{2XY}^S r_1 \lhd_{2XY}^S r_{10}
 \lhd_{2XY}^S r_9  \lhd_{2XY}^S r_7.
\end{multline*}
Observe that rankings based on $\unlhd_{Lex1}^S$, $\unlhd_{XY}^S$ and $\unlhd_{2XY}^S$ are the same and differ from the given by $\unlhd_{Lex2}^S$ in the position of $r_1$ and $r_{13}$. So, in the case of Table \ref{tab-cidades}, the minimum route to travel the capitals of the Brazilian Northeast of the Travelling salesman problem is $r_{17}$.

\section{Final Remarks}

In this paper, we  generalize the notion of admissible order on the set of closed subintervals of $[0,1]$ to the set of fuzzy numbers equipped with an arbitrary order. Although the Klir and Yuan order is not consensually accepted as the natural order for the set of fuzzy numbers, most of the orders proposed for fuzzy numbers refine this order. So we deal {with} the Klir-Yuan order as the ``natural'' one for $\NFR$ and explore the admissible order with respect to this order.

Applications of admissible orders on several domains {have} been succesfully developed in several areas as can be seen in \cite{Annax20,Laura17,Bentkowska15,BustinceGBKM13,Laura16b} and the same happens with the application of fuzzy numbers. {Thus}, it may be expected that in a future efforts can  be made to develop interesting applications of admissible orders on $\NFR$.

In \cite{bustince2013generation} a construction method of admissible orders over the set of closed subintervals of $[0,1]$ based on  aggregation functions is provided  and lately generalized in \cite{Santana2020}.  As a future work, we will intend to introduce a generation method of admissible orders on  $\NFR$.

Namely, in \cite{Cornelis2004} is one of the many proposals in the literature regarding weighted fuzzy graphs. Our use of admissible orders is what guarantees that the final cost of all possible paths can be ordered linearly, thus always obtaining the shortest way.

\section*{Acknowledgments}  This work was  supported  by the Brazilian funding agency  \textbf{CNPq (Brazilian Research Council)} under Projects: 311429/2020-3 and by the project PID2019-108392GB-I00 (AEI/10.13039/ 501100011033) of the Spanish Government.

\ifCLASSOPTIONcaptionsoff
  \newpage
\fi

\begin{IEEEbiography}[{\includegraphics[width=1in,height=1.25in,clip,keepaspectratio]{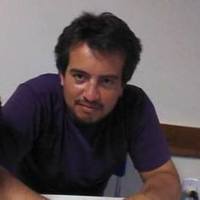}}]{Nicol\'as Zumelzu}
received the M.Sc. degree  from the Federal University of Pernambuco, Brazil, in mathematics (2017). He is currently an Assistant Professor with the University of Magallanes, Punta Arenas, Chile. His research interests include Partial differential equations, Integral equations and Fuzzy mathematics.
\end{IEEEbiography}

\begin{IEEEbiography}[{\includegraphics[width=1in,height=1.25in,clip,keepaspectratio]{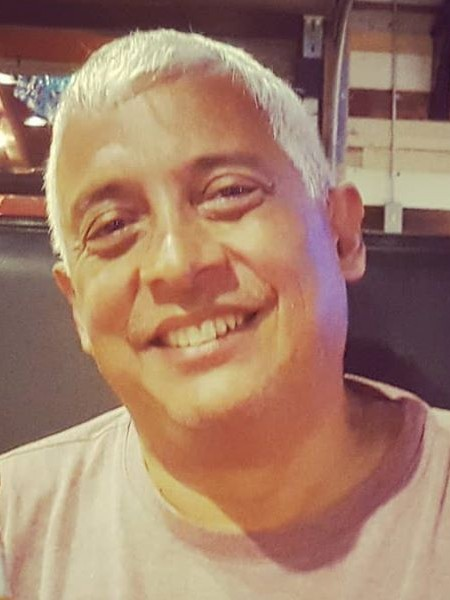}}]{Benjam\'in Bedregal} (M'16) received the Ph.D. degree in computer sciences from the Federal University of Pernambuco, Recife, Brazil, in 1987 and 1996, respectively.
In 1996, he became an Assistant Professor with the Department of Informatics and Applied Mathematics, Federal University of Rio Grande do Norte, Natal, Brazil, where he is currently a Full Professor. His research interests include nonstandard fuzzy sets theory, aggregation and
preaggregation functions, clustering, fuzzy mathematics, and fuzzy automata.
\end{IEEEbiography}
\begin{IEEEbiography}[{\includegraphics[width=1in,height=1.25in,clip,keepaspectratio]{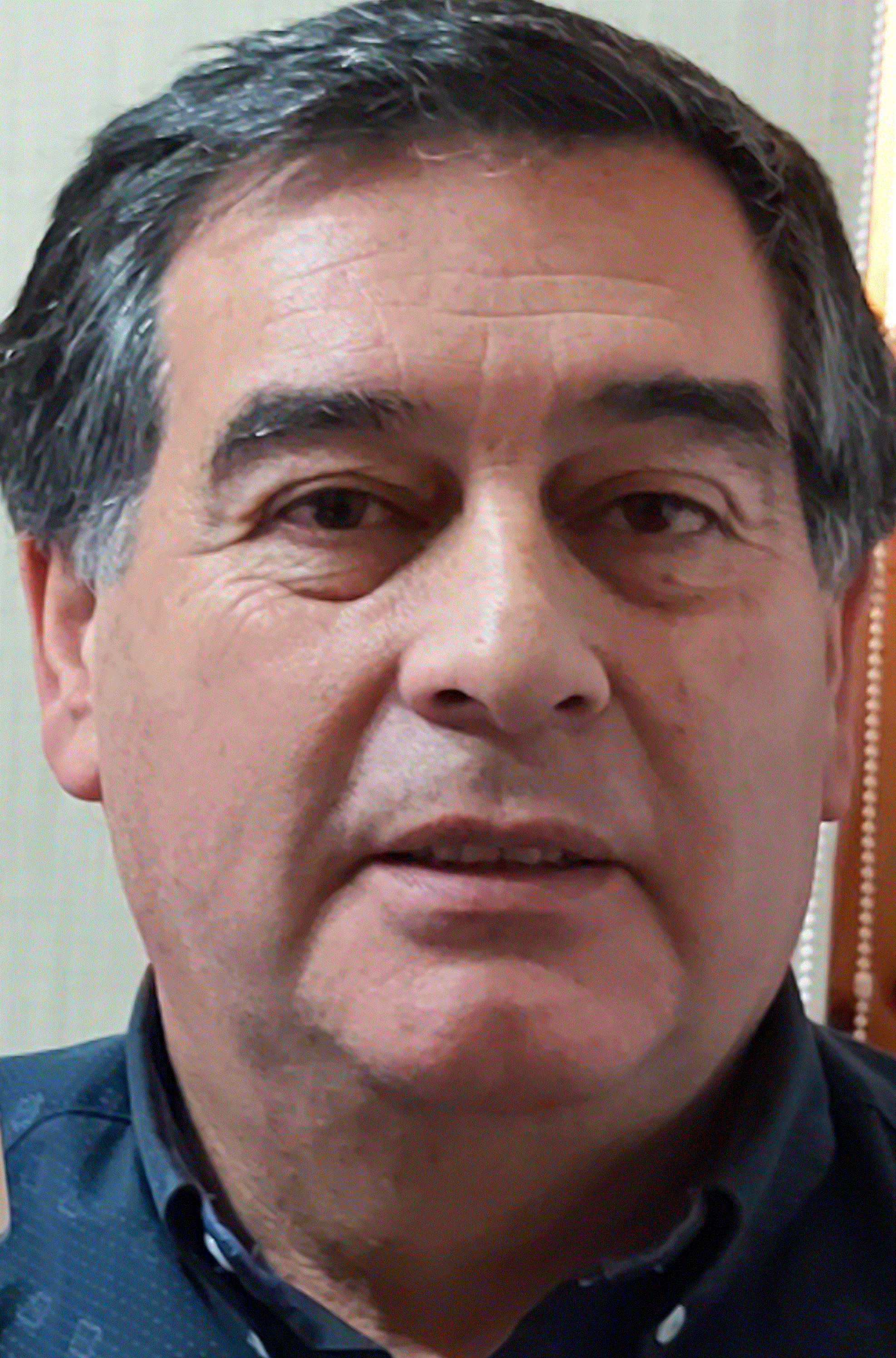}}]{Edmundo Mansilla} (M'18)
 received the Ph.D. degree in mathematic from the Federal University of Pernambuco, Recife Brazil (1999). He is a Associated Professor with the of University of Magallanes, Punta Arenas, Chile. His research interests include Celestial mechanic, Dinamical System and fuzzy mathematics.
\end{IEEEbiography}
\begin{IEEEbiography}[{\includegraphics[width=1in,height=1.25in,clip,keepaspectratio]{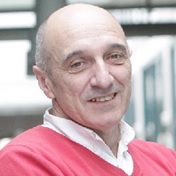}}]
{Humberto Bustince} (M’08--SM'15) Humberto Bustince received the
Ph.D. degree in mathematics from the Public University
of Navarra, Pamplona, Spain, in 1994.
He is currently a Full Professor with the Department
of Automatics and Computation, Public University
of Navarra. He is the author of more than 200 papers
published original articles in ISI journals and is involved in teaching
artificial intelligence for students of computer
sciences. His research interests include fuzzy logic
theory, extensions of fuzzy sets (type-2 fuzzy sets,
interval-valued fuzzy sets, Atanassovs intuitionistic
fuzzy sets), fuzzy measures, aggregation functions, and fuzzy techniques for
image processing.
\end{IEEEbiography}

\begin{IEEEbiography}[{\includegraphics[width=1in,height=1.25in,clip,keepaspectratio]{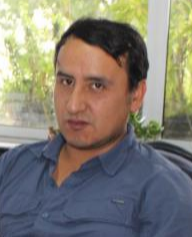}}]{Roberto D\'iaz} (M'18)
Ph.D. degree in Applied Mathematics from the B\'io-B\'io University, Concepci\'on, Chile (2018). He is a Assistant Professor with the Departamento de Ciencias Exactas, Universidad de los Lagos, since 2016. His research interests Partial differential equations, Integral equations, General relativity and Fuzzy mathematics.
\end{IEEEbiography}
\end{document}